%% file: Transactions2.tex
\documentclass{amsart}
\usepackage{amsmath,amsfonts,amssymb}
\usepackage[margin=30mm]{geometry}
\usepackage[all]{xy}
\usepackage{comment}
\usepackage{enumitem}
\usepackage{graphicx}
\usepackage{caption}
\usepackage{subcaption}
\usepackage{xcolor}
\usepackage{url}
\usepackage[hidelinks,colorlinks=true,linkcolor=blue,urlcolor=black, citecolor=black,linktocpage=true]{hyperref}
\usepackage{tikz}
\usepackage{pgfplots}
\usepackage{thmtools}
\usepackage{thm-restate}
\usepackage[capitalize,noabbrev]{cleveref}

\def\cH{\mathcal{H}}

\def\Q{\mathbb{Q}}
\def\R{\mathbb{R}}
\def\Z{\mathbb{Z}}
\def\RP{\mathbb{R}P}
\def\Id{\text{Id}}

\def\ol{\overline}
\def\sm{\setminus}
\def\lk{\text{lk}}
\def\alphas{\boldsymbol{\alpha}}
\def\betas{\boldsymbol{\beta}}
\def\SW{\text{SW}}
\def\Spin{\text{Spin}}
\def\fs{\mathfrak{s}}
\DeclareMathOperator{\Int}{Int}

\DeclareMathOperator{\Hom}{Hom}

\newtheorem{thm}{Theorem}[section]
\newtheorem{cor}[thm]{Corollary}
\newtheorem{lem}[thm]{Lemma}

\newtheorem{prop}[thm]{Proposition}
\theoremstyle{definition}
\newtheorem{defn}[thm]{Definition}
\theoremstyle{definition}
\newtheorem{rem}[thm]{Remark}

\newtheorem{conj}[thm]{Conjecture}

\begin{document}

\title{Examples of topologically unknotted tori}

	\author{Andr\'{a}s Juh\'{a}sz}
	\address{Mathematical Institute, University of Oxford, Andrew Wiles Building,
		Radcliffe Observatory Quarter, Woodstock Road, Oxford, OX2 6GG, UK}
	\email{juhasza@maths.ox.ac.uk}
	
	\author{Mark Powell}
	\address{School of Mathematics and Statistics, University of Glasgow, University Place, Glasgow, G12 8QQ, UK}
	\email{mark.powell@glasgow.ac.uk}
	
	\def\subjclassname{\textup{2020} Mathematics Subject Classification}
	\expandafter\let\csname subjclassname@1991\endcsname=\subjclassname
	
	\begin{abstract}
		We show that certain smooth tori with group $\Z$ in $S^4$ have exteriors with standard equivariant intersection forms, and so are topologically unknotted. These include the turned 1-twist-spun tori in the 4-sphere constructed by Boyle, the union of the genus one Seifert surface of Cochran and Davis that has no slice derivative with a ribbon disc, and tori with precisely four critical points whose middle level set is a 2-component link with vanishing Alexander polynomial. This gives evidence towards the conjecture that all $\Z$-surfaces in $S^4$ are topologically unknotted, which is open for genus one and two. It is unclear whether these tori are smoothly unknotted, except for tori with four critical points whose middle level set is a split link. The double cover of $S^4$ branched along any of these surfaces is a potentially exotic copy of $S^2 \times S^2$, and, in the case of turned twisted tori, we show they cannot be distinguished from $S^2 \times S^2$ using Seiberg--Witten invariants.
	\end{abstract}
	
	\maketitle
	
	\section{Introduction}
	We study  the topological isotopy problem for locally flat embeddings of closed, orientable surfaces in $S^4$. Let $\Sigma \subseteq S^4$ be such a surface of genus $g$.  If $\pi_1(S^4 \setminus \Sigma) \cong \Z$, generated by a meridian of $\Sigma$, then we say that $\Sigma$ is a \emph{$\Z$-surface}. The surface $\Sigma$ is \emph{topologically unknotted} if it is topologically isotopic to the standard genus $g$ surface in $S^3$, i.e.\  if it topologically bounds a genus $g$ handlebody. If $\Sigma$ has genus one, then we say that $\Sigma$ is a $\Z$-torus.
	Conway and the second author~\cite{Conway-Powell} have shown that any locally flat, embedded, closed, orientable $\Z$-surface $\Sigma \subseteq S^4$ with $g(\Sigma) \not\in \{1, 2\}$ is topologically unknotted. The topological unknotting conjecture states that this also holds for $g \in \{1,2\}$.

Using the topological classification of $\Z$-surfaces in terms of the equivariant intersection forms of their exteriors due to Conway and the second author~\cite{Conway-Powell}, we show that certain $\Z$-tori are topologically unknotted.  Our three main results, for the three types of tori we investigate, are described in Sections~\ref{subsection-intro-turned}, \ref{subsection-intro-pushed}, and \ref{subsection-intro-crit} below.

We say that a pair of smoothly embedded surfaces in a 4-manifold are \emph{exotic} if they are topologically but not smoothly isotopic. In $S^4$, this is equivalent to saying that they are orientation-preserving ambiently homeomorphic but not orientation-preserving ambiently diffeomorphic.  Finashin, Kreck, and Viro~\cite{FKV} constructed infinitely many exotic copies of $\# 10 \RP^2$ in $S^4$. This was later improved by Finashin~\cite{Finashin} to $\# 6 \RP^2$ in $S^4$. These are distinguished by their double branched covers.
Further improvements were recently announced by Miyazawa~\cite{miyazawa2023gauge} for $\RP^2$ and Mati\'c, \"Ozt\"urk, and Stipsicz~\cite{matic2023exotic} for $5\RP^2$.
However, there is no example known of an exotic pair of closed, orientable surfaces in $S^4$.

This work can be viewed as collecting evidence towards the genus one topological unknotting conjecture. On the other hand, the examples we consider are potentially exotic surfaces, i.e.\ so far as we know, they could be counterexamples to the smooth unknotting conjecture for tori; see \cref{sec:dbc}. Furthermore, if we take the double cover of $S^4$ branched along any of these tori $T$, we obtain a potentially exotic copy of $S^2 \times S^2$.

\subsection{Turned twisted tori}\label{subsection-intro-turned}
	Given a knot $K$ in $S^3$, Artin~\cite{Artin-spinning} associated to it a 2-knot (i.e.\  an embedded 2-sphere) in $S^4$ using a construction called \emph{spinning}. This was later extended by Zeeman~\cite{Zeeman-twisting} to \emph{twist-spinning}, and then by Litherland~\cite{Litherland-deform-spinning} to \emph{deform-spinning}.
	Building on this, in 1993, Boyle~\cite{Boyle} associated to a knot $K$ in $S^3$ and integers $i$, $j$ the \emph{$i$-turned $j$-twist-spun torus} in $S^4$, which we will denote by~$T_{K,i,j}$. See Section~\ref{sec:twist-spun-tori} for the definition. The smooth isotopy class of~$T_{K,i,j}$ only depends on the parity of $i$. At the end of his paper, Boyle noted that the torus~$T_{K,1,1}$ is a $\Z$-torus and asked whether it is standard.
	
	\begin{restatable}{thm}{Boyle}\label{thm:Boyle}
	For every knot $K$ in $S^3$, the turned 1-twist-spun torus $T_{K,1,1}$ is topologically unknotted.
	\end{restatable}
	
	We also observe in \cref{sec:dbc} that the double cover of $S^4$ branched along $T_{K,1,1}$ is $(S^2 \times S^2) \# X$ for a homotopy 4-sphere $X$. So one cannot use the Seiberg--Witten invariants of the double branched covers to smoothly distinguish $T_{K,1,1}$ and the standard torus $T^2 \subseteq S^3 \subseteq S^4$.
	

\subsection{Seifert surface union a ribbon disc}\label{subsection-intro-pushed}
	Recall that, given a knot $K$ in $S^3$, a \emph{homotopy ribbon disc} $D$ for $K$ is a locally flat disc embedded in $D^4$ such that $\partial D = K$ and the map
	\[
	\pi_1(S^3 \setminus K) \to \pi_1(D^4 \setminus D)
	\]
	is surjective. In particular, every ribbon disc is homotopy ribbon; see Gordon~\cite{Gordon-ribbon}.
	Let $K$ be a knot in $S^3$ with Seifert surface $F$ and homotopy ribbon disc $D$. Then $\Sigma := F \cup D \subseteq S^4$ is a $\Z$-surface.
	
	A \emph{slice derivative} for a genus $g$ Seifert surface $F$ for a knot $K$ in $S^3$ is a collection of $g$ pairwise disjoint, homologically linearly independent simple closed curves that form a smoothly slice link and have self-linking zero~\cite{CHL-derivatives}. If $F$ has a slice derivative, we can perform surgery on $F$ in $D^4$ along a slice derivative and obtain a smooth slice disc $D$ for $K$. Then $F \cup D$ is smoothly unknotted; see Lemma~\ref{lem:slice-derivative}. If $D$ is a ribbon disc that is not obtained from a Seifert surface $F$ by surgering along a slice derivative, Gabai has asked whether $F \cup D$ is topologically or smoothly unknotted.
	
	Kauffman conjectured that every genus one Seifert surface for a slice knot admits a slice derivative.
	Cochran and Davis~\cite{Cochran-Davis} have shown that there exists a ribbon knot admitting a genus one Seifert surface $F$ that has no slice derivative. Moreover, there is an example where $F$ is the unique minimal genus one Seifert surface, up to isotopy. The union of $F$ and a ribbon disc is therefore a potentially smoothly knotted $\Z$-torus in $S^4$. We show that it is topologically unknotted.
 	
	\begin{restatable}{thm}{ribbon}\label{thm:ribbon}
		Let $K$ be the knot in $S^3$ with genus one Seifert surface $F$ and ribbon disc $D$ constructed by Cochran and Davis. Then $\Sigma := F \cup D$ is topologically unknotted.
	\end{restatable}
	
In \cref{prop:alex-poly-one-torus}, we note that a genus one Seifert surface union a slice disc for a knot $K$ is a topologically unknotted $\Z$-torus also when $\Delta_K\doteq 1$.

We will show in Section~\ref{sec:dbc} that the double cover $M$ of $S^4$ branched along $\Sigma$ can be obtained from $S^2 \times S^2$ by knot surgery along a 0-homologous torus. By Theorem~\ref{thm:ribbon}, the 4-manifolds $M$ and $S^2 \times S^2$ are homeomorphic. We do not know whether $M$ is diffeomorphic to $S^2 \times S^2$ or even whether it is irreducible. One cannot distinguish $M$ from $S^2 \times S^2$ using the knot surgery formula of Fintushel and Stern~\cite{FSKnotSurgery}, since it requires the torus to be homologically essential.

Theorem~\ref{thm:ribbon} follows from the following result on knot surgery along a torus in the complement of an unknotted surface in $S^4$.

\begin{restatable}{thm}{knotsurgery}\label{thm:knot-surgery}
	Let $\Sigma \subseteq S^4$ be a smoothly embedded, topologically unknotted, oriented surface of genus~$g$. Furthermore, let $T \subseteq S^4 \setminus N(\Sigma)$ be a smoothly embedded torus that is unknotted in $S^4$, and let $J$ be a knot in $S^3$.
	Suppose that the push-off $s(m)$, for some primitive, essential curve $m$ on $T$, is 0-homotopic in $S^4 \setminus (\Sigma \cup T)$. Then the $t$-twist, $r$-roll knot surgery $\Sigma_{t, r}(T, J)$ is topologically unknotted for all $t$, $r \in \Z$.
\end{restatable}

For further details on the $t$-twist $r$-roll knot surgery construction involved, see Section~\ref{sec:knot-surgery}.

\subsection{Tori with four critical points}\label{subsection-intro-crit}

Let $\Sigma \subseteq S^4$ be a smoothly embedded torus. Consider $S^4 = S^4_+ \cup S^4_-$, where $S^4_{\pm} \cong D^4$.  Let $r \colon D^4 \to I$ be the radial function, and let $f_{\pm} \colon S^4_{\pm} \to \R$ be given by $\pm(1-r)$. These glue to give a standard Morse function $f \colon S^4 \to \R$, with precisely two critical points $f^{-1}(\pm 1)$ at the centre of $S^4_{\pm}$.  Assume that $f|_\Sigma \colon \Sigma \to \R$ is a Morse function with precisely four critical points. Then we say that $\Sigma$ \emph{has four critical points}.
In this case, the exterior of $\Sigma$ has a handle decomposition with a unique 1-handle, so $\Sigma$ is a $\Z$-torus.

We can isotope $\Sigma$ so that $\Sigma \cap S^4_-$ contains the minimum of $f|_{\Sigma}$ and one critical point of index one, while $\Sigma \cap S^4_+$ contains the maximum of $f|_{\Sigma}$ and one critical point of index one. Then $L:= \Sigma \cap \partial S^4_{\pm}$ is a 2-component link in $\partial S^4_{\pm} = S^3$.

\begin{restatable}{thm}{morse}\label{thm:morse}
Let $\Sigma \subseteq S^4$ be a smoothly embedded torus with four critical points. Suppose that the single-variable Alexander polynomial $\Delta_L$ of the 2-component link $L := \Sigma \cap S^3$ is zero, where $S^3$ is the equator of $S^4$. Then $\Sigma$ is topologically unknotted.
\end{restatable}

This answers a special case of Problem~4.30 of Kirby~\cite{Kirby-problems} in the topological category. Scharlemann~\cite{Scharlemann1985} has shown that every 2-sphere in $S^4$ with four critical points is smoothly unknotted, and Bleiler--Scharlemann~\cite{BleilerScharlemann} have shown that every embedded $\mathbb{RP}^2 \subseteq S^4$ with three critical points is smoothly unknotted. As far as we know, it is open whether every torus with four critical points is smoothly unknotted.
We are not aware of any explicit potentially smoothly knotted examples associated with \cref{thm:morse}. We show the following result about smooth unknotting.

\begin{restatable}{thm}{four}\label{thm:four}
	Let $\Sigma \subseteq S^4$ be a torus with four critical points, and suppose that the link $L := \Sigma \cap S^3$ is split, where $S^3$ is the equator of $S^4$. Then $\Sigma$ is smoothly unknotted.
\end{restatable}

When $L$ is not split, it is not even clear whether the double cover of $S^4$ branched over $\Sigma$ contains a self-intersection zero sphere, so potentially it could be distinguished from $S^2 \times S^2$ using Seiberg--Witten invariants.

\begin{rem}
  An interesting family of examples of $\Z$-surfaces that would be good candidates for potential future investigation are the following. Let $K$ be a knot in the 3-sphere with two non-isotopic genus one Seifert surfaces $F_+$ and $F_-$. Consider $S^4 = S^4_+ \cup_{S^3} S^4_-$, where $S^4_{\pm}$ is a copy of $D^4$, and push the interior of $F_{\pm} \subseteq S^3$ into $S^4_{\pm}$. The result $F_+ \cup_K F_-$ is a genus two $\Z$-surface. Is it topologically, or even smoothly unknotted?
\end{rem}

\subsection*{Notation} Throughout, we write
	\[
	D^3_r := \{v \in \R^3 : |v| \le r\} \text{ and } S^2_r := \{v \in \R^3  :  |v| = r\}.
	\]
	Let $B^3_r := \Int(D^3_r)$. As usual, $D^3 := D^3_1$, $S^2 := S^2_1$, and $B^3 := B^3_1$. If $A$ is a submanifold of $B$, then we write $N(A)$ for an open tubular neighbourhood of $A$ in $B$ and $\nu_{A \subseteq B}$ for the normal bundle of $A$ in $B$. For smooth manifolds $M$ and $N$, we write $M \cong N$ if they are diffeomorphic. Finally, let
	\[
	\R^3_+ := \{(x,y,z) \in \R^3  :  z \ge 0\} \text{ and } \R^3_{>0} := \{(x,y,z) \in \R^3  :  z > 0\}.
	\]

\subsection*{Organisation}

In \cref{section:std-int-form-implies-unknotted}, we set out the strategy that we will use for all our proofs, which is to compute the equivariant intersection pairing of the surface knot exterior, and apply \cite{Conway-Powell}.
In \cref{sec:twist-spun-tori}, we recall the construction of turned, twisted tori.
Then, in \cref{section:proof-turned-tori-unknotted}, we show that they are topologically unknotted, proving \cref{thm:Boyle}. In Section~\ref{sec:knot-surgery}, we prove \cref{thm:knot-surgery} on knot surgeries.
In \cref{sec:Cochran-Davis-examples}, we  prove \cref{thm:ribbon}, and in \cref{section:four-crit-points}, we prove Theorems~\ref{thm:morse} and~\ref{thm:four}.
In \cref{sec:dbc}, we observe that using the double branched cover to detect that some turned twisted torus is exotic would entail finding a counterexample to the smooth 4-dimensional Poincar\'{e} conjecture. Furthermore, we identify the double branched cover of the union of the genus one Seifert surface and the ribbon disc constructed by Cochran and Davis as knot surgery on $S^2 \times S^2$ along a 0-homologous torus.

	\subsection*{Acknowledgements} We would like to thank Anthony Conway, Chris Davis, David Gabai, Jason Joseph, Maggie Miller, Andr\'as Stipsicz, Zolt\'an Szab\'o, and Ian Zemke for helpful discussions and suggestions. We are also particularly grateful to Brendan Owens, who pointed out a mistake in our proof of an earlier version of \cref{thm:morse}, and to an anonymous referee for several very helpful suggestions on the exposition.
	MP was partially supported by EPSRC New Investigator grant EP/T028335/2 and EPSRC New Horizons grant EP/V04821X/2. For the purposes of open access, the authors have applied a CC-BY public copyright license to any Author Accepted Manuscript arising from this submission.
	
	\section{Standard intersection form implies unknotted}\label{section:std-int-form-implies-unknotted}
	
	For a $\Z$-surface $\Sigma \subseteq S^4$, let
	\[
	E_\Sigma := S^4 \setminus N(\Sigma)
	\]
	be the exterior of $\Sigma$.
	Let $\Lambda$ be the group ring $\Z[\Z]$, which is isomorphic to the ring of Laurent polynomials $\Z[t, t^{-1}]$. This admits an involution determined by $t \mapsto t^{-1}$. Consider the Hermitian form
	\[
	\cH_2 :=  \left( \Lambda^2,
	\begin{pmatrix}
		0 & 1 - t \\
		1 - t^{-1} & 0
	\end{pmatrix} \right).
	\]
	The following is a consequence of \cite[Theorem~1.4]{Conway-Powell} and \cite[Lemma~6.1]{Conway-Powell}.
	
	\begin{thm}\label{thm:intersection-form}
		Let $\Sigma \subseteq S^4$ be a $\Z$-surface of genus $g$. Then the $\Lambda$-intersection form of $E_\Sigma$ is isometric to
		\[
		\bigoplus_g \cH_2
		\]
	    if and only if $\Sigma$ is topologically unknotted.
	\end{thm}
	
	So, to prove Theorems~\ref{thm:Boyle}, \ref{thm:ribbon}, and~\ref{thm:morse}, it suffices to compute the equivariant intersection pairings of the surface exterior.

	Next, we record some homological facts about the exteriors of $\Z$-surfaces in $S^4$. If $\Sigma \subseteq S^4$ is an embedded surface and $\gamma \subseteq \Sigma$ is a simple closed curve, then we write $T_\gamma$ for the total space of the restriction of the unit normal circle bundle of $\Sigma$ to $\gamma$, which is known as the \emph{rim torus} over $\gamma$.
	
	\begin{prop}\label{prop:homology-closed}
		Let $\Sigma \subseteq S^4$ be a closed, connected, and oriented $\Z$-surface of genus $g$. Then
		\begin{enumerate}[font=\normalfont, label = (\roman*), ref= \roman*]
			\item \label{it:H0-lambda} $H_0(E_\Sigma; \Lambda) \cong \Z$;
			\item \label{it:H1-lambda} $H_1(E_\Sigma; \Lambda) = 0$;
			\item \label{it:boundary-closed}  $H_1(\partial E_{\Sigma}; \Lambda) \cong H_1(\Sigma \times S^1; \Lambda) \cong (\Lambda / (1 - t))^{\oplus 2g}$;
			\item \label{it:H2} $H_2(E_\Sigma; \Z) \cong \Z^{2g}$, generated by rim tori;
			\item \label{it:H2-lanbda} $H_2(E_\Sigma; \Lambda) \cong \Lambda^{2g}$, generated by generically immersed spheres that map to the classes of the rim tori under the surjective map $H_2(E_{\Sigma};\Lambda) \to H_2(E_{\Sigma};\Z)$ induced by tensoring with $\Z$ over~$\Lambda$;
			\item \label{it:H3-lambda} $H_3(E_\Sigma; \Lambda) = 0$.
		\end{enumerate}
	\end{prop}
	
	\begin{proof}
		Part~\eqref{it:boundary-closed}
		is a special case of \cite[Lemma~5.5]{Conway-Powell}. 	Part~\eqref{it:H2} follows from Alexander duality.
All the other parts follow from \cite[Lemma~3.2]{Conway-Powell}, except for \eqref{it:H2-lanbda}, where some more argument is required, which we now give. In \cite[Lemma~3.2]{Conway-Powell}, it is shown that $H_2(E_{\Sigma};\Lambda)$ is a free $\Lambda$-module. That it has rank $2g$ follows from considering $\Q(t)$ coefficients: to see this, note that, by \eqref{it:H0-lambda}, \eqref{it:H1-lambda}, and \eqref{it:H3-lambda}, the Euler characteristic over $\Q(t)$ equals $\operatorname{rk} H_2(E_\Sigma;\Lambda)$, whereas computing over $\Q$ shows that $\chi(E_{\Sigma})=2g$.  Hence $\operatorname{rk} H_2(E_\Sigma;\Lambda)=2g$.

The elements of the standard basis $\{e_i\}_{i=1}^{2g}$ of $H_2(E_\Sigma;\Lambda) \cong \Lambda^{2g}$  are represented by generically immersed spheres (i.e.\ locally a flat embedding and all self-intersections are transverse double points) in the smooth 4-manifold $\Int(E_\Sigma)$ (it is smooth because it is an open subset of $S^4$), which we can see as follows. Apply the Hurewicz theorem to the universal cover of $E_\Sigma$ to see that $\pi_2(E_\Sigma) \cong H_2(E_\Sigma;\Lambda)$. Then apply the density of smooth immersions in $C^0(S^2, \Int(E_\Sigma))$~\cite[Theorems~2.2.6~and~2.2.12]{Hirsch-differential-topology}, and finally the density of generic immersions in the space of smooth immersions~\cite[Chapter~III, Corollary~3.3]{GoGu}. Since the $\{e_i\}_{i=1}^{2g}$ are represented by generically immersed spheres, so are  $\{e_i \otimes 1\}_{i=1}^{2g} \subseteq H_2(E_{\Sigma};\Lambda)\otimes_{\Lambda} \Z \cong  \Lambda^{2g} \otimes_{\Lambda} \Z$.

  It was shown in the proof of \cite[Lemma~5.10]{Conway-Powell} that the canonical map
  \[
  H_2(E_{\Sigma};\Lambda)\otimes_{\Lambda} \Z \xrightarrow{\cong} H_2(E_{\Sigma};\Z)
  \]
  is an isomorphism.  It follows that
  $H_2(E_{\Sigma};\Lambda) \to H_2(E_{\Sigma};\Z)$ is indeed surjective.  Moreover, rim tori generate the codomain $H_2(E_{\Sigma};\Z) \cong \Z^{2g}$ by \eqref{it:H2}.  The inverse images of these rim tori in $H_2(E_{\Sigma};\Lambda)\otimes_{\Lambda} \Z$ are linear combinations of the basis elements $\{e_i \otimes 1\}_{i=1}^{2g}$, which we can write as
  \[
  \sum_j a_{ij} (e_j \otimes 1) = \biggl(\sum_j a_{ij} e_j \biggr) \otimes 1
  \]
  for some $a_{ij} \in \Z$ such that the matrix $(a_{ij})_{i,j=1}^{2g}$ is invertible over $\Z$.
  Hence, the elements $\{e_i' := \sum_j a_{ij} e_j\}_{i=1}^{2g}$ generate $H_2(E_\Sigma;\Lambda)$, because $(a_{ij})_{i,j=1}^{2g}$ is also invertible over $\Lambda$. Furthermore, the $e_i'$ map to the rim tori.
  By taking $a_{ij}$ parallel copies of $e_j$ for $j \in \{1, \dots, 2g\}$, linearly ordering all these spheres in an arbitrary way, and connecting each sphere with the next using a tube~\cite[Chapter~11]{FNOP} around an embedded framed arc that is coherent with the orientations of the spheres, we obtain a generically immersed sphere in $\Int(E_\Sigma)$ representing $e_i'$.
\end{proof}

\begin{defn}
  Let $E$ be a compact, connected, oriented, and based 4-manifold with $\pi_1(E) \cong \Z$. We say that a closed, connected, and oriented surface $S \looparrowright E$ immersed in $E$, and equipped with a basing arc from $S$ to the basepoint of $E$,  is \emph{$\Z$-trivial} if the inclusion-induced map $\pi_1(S) \to \pi_1(E)$ is the trivial homomorphism.
\end{defn}

The next lemma encapsulates the strategy for the proofs of most of our theorems.

\begin{lem}\label{lemma:main-strategy}
	Let $\Sigma \subseteq S^4$ be a closed, connected, and oriented $\Z$-surface of genus $g$.  Let $S_1,\dots, S_{2g} \subseteq E_\Sigma$ be generically immersed $\Z$-trivial surfaces, and let $\lambda$ be the $\Lambda$-intersection form on $H_2(E_\Sigma;\Lambda)$. Suppose that
	\[
	A:= [\lambda([S_i], [S_j])]_{i,j \in \{1,\dots, 2g\}} = \bigoplus_{k=1}^g \begin{pmatrix}
		0 & 1-t \\
		1-t^{-1}  & d_k
	\end{pmatrix}
	\]
	for some $d_k \in \Lambda$. Then $\{[S_1], \dots, [S_{2g}]\}$ is a basis for $H_2(E_\Sigma;\Lambda)\cong \Lambda^{2g}$ and $\lambda$ is isometric to $\bigoplus_g \mathcal{H}_2$. Hence $\Sigma$ is topologically unknotted by \cref{thm:intersection-form}.
\end{lem}

\begin{proof}
	Let $\langle e_1, \dots, e_{2g} \rangle$ be a  basis for $H_2(E_{\Sigma}; \Lambda) \cong \Lambda^{2g}$. Let $B$ be a matrix over $\Lambda$ such that
	\begin{equation}\label{eqn:B}
		([S_1], \dots, [S_{2g}]) = B \cdot (e_1, \dots, e_{2g}),
	\end{equation}
	i.e.\  $[S_i] = Be_i$ for $i \in \{1, \dots, 2g\}$.
	By part~\eqref{it:boundary-closed} of Proposition~\ref{prop:homology-closed}, the order of $H_1(\partial E_\Sigma; \Lambda)$ is $(1-t)^{2g}$, which, up to units in $\Lambda$, equals $(1-t)^g(1 - t^{-1})^g$. Hence, up to units,
	\begin{equation}\label{eqn:order}
		\det \left([\lambda(e_i, e_j)]_{i, j \in \{1,2\}} \right) \doteq (1-t)^g(1 - t^{-1})^g.
	\end{equation}
	By a computation, the hypothesis, equation~\eqref{eqn:B}, another computation, and finally using equation~\eqref{eqn:order},  we obtain that
	\[
	\begin{split}
		(-1)^g (1-t)^g(1 - t^{-1})^g
&=\det A \\
&=  \det \left([\lambda([S_i], [S_j])]_{i, j \in \{1,\dots,2g\}}\right) \\
&=  \det \left([\lambda(Be_i, Be_j)]_{i, j \in \{1,\dots,2g\}}\right)  \\
&=  \det(B)  \det \left([\lambda(e_i, e_j)]_{i, j \in \{1,\dots,2g\}} \right) \det \big( \overline{B^T} \big) \\
&\doteq \det(B) \overline{\det(B)} (1-t)^g(1 - t^{-1})^g.
	\end{split}
	\]
The last equality is again up to units, so we can ignore the $(-1)^g$ factor, and deduce that
\[(1-t)^g(1 - t^{-1})^g \doteq \det(B) \overline{\det(B)} (1-t)^g(1 - t^{-1})^g.\]
	Since $\Lambda$ is an integral domain, it follows that $\det(B) = \pm t^k$ for some $k \in \Z$, so $B$ is invertible over $\Lambda$, and hence $\{ Be_1, \dots, Be_{2g} \} = \{[S_1], \dots, [S_{2g}]\}$ is a basis, as claimed.
This means that the $\Lambda$-intersection form $\lambda$ of $E_{\Sigma}$ is represented by $A$.

Augmenting this matrix with $\varepsilon \colon \Lambda \to \Z$ gives the $\Z$-valued intersection form
\[
		\bigoplus_{k=1}^g \begin{pmatrix}
			0 & 0 \\
			0  & \varepsilon(d_k)
		\end{pmatrix}
\]
of $E_{\Sigma}$.
By Proposition~\ref{prop:homology-closed}, $H_2(E_\Sigma;\Z)$ is generated by rim tori, which lie in the image of $H_2(\partial E_\Sigma;\Z)$, and so all intersection numbers involving them vanish. It follows that $\varepsilon(d_k) =0$ for all $k$.

Now fix $k \in \{1, \dots,g\}$.
Since $\lambda$ is Hermitian,
\[d_k= a_0 + a_1(t+t^{-1}) + \cdots + a_n(t^n + t^{-n})\]
for some $n \in \mathbb{N}_0$ and for some integers $a_0,\dots,a_n$.
Since $\varepsilon(d_k)=0$, we deduce that \[a_0 +2a_1 + \cdots + 2a_n=0.\]
For any $p(t) \in \Lambda$,  we can add $(1-t)p(t) + (1-t^{-1})p(t^{-1})$ to $d_k$ while keeping the other three entries of the matrix
\[
	\begin{pmatrix}
		0 & 1-t \\
		1-t^{-1}  & d_k
	\end{pmatrix}
\]
unchanged by the change of basis $\{[S_{2i}], [S_{2i+1}]\} \mapsto \{[S_{2i}], [S_{2i+1}] + p(t) [S_{2i}]\}$ of $H_2(E_\Sigma;\Lambda)$.
Applying this with $p(t) = a_nt^{n-1}$, we replace $d_k$ with
\[a_0 + a_1(t+t^{-1}) + \cdots + (a_{n-1} + a_n)\bigl(t^{n-1} + t^{-(n-1)}\bigr).\]
Iterate this change, at each stage removing the terms whose exponents have the highest absolute value, to obtain
\[a_0 + (a_1 + \cdots + a_n)(t+t^{-1}).\]
One more basis change, with $p(t) = (a_1 + \cdots + a_n)t$, yields $a_0 +2a_1 + \cdots + 2a_n =0$.
Thus, by a basis change, we can replace $d_k$ with $0$.

Performing these basis changes for each $k = 1, \dots,g$,  we find that $\lambda$, the $\Lambda$-intersection form of $E_{\Sigma}$, is isometric to $\cH_2$.  It follows  by Theorem~\ref{thm:intersection-form} that $\Sigma$ is topologically unknotted.
\end{proof}

In each of Sections~\ref{section:proof-turned-tori-unknotted}, \ref{sec:Cochran-Davis-examples}, and \ref{section:four-crit-points}, we will exhibit $\Z$-trivial surfaces that generate a submodule with $\Lambda$-intersection form represented by $A$, and then we will appeal to \cref{lemma:main-strategy}.

It will sometimes be helpful to compute the homological self-intersection numbers $\lambda([S_i],[S_i]) \in \Lambda$ in Lemma~\ref{lemma:main-strategy} using the Wall self-intersection number $\mu(S_i)$, so we recall this next; see~\cite[Chapter~5]{Wall-surgery-book}.
Let $S \subseteq E_\Sigma$ be a $\Z$-trivial surface that is the image of a generic immersion into $E_{\Sigma}$.
For each double point $p \in S$, we obtain a sign $\varepsilon_p$ and a double point loop $\gamma_p$ in $S$, with the latter well-defined up to orientation and homotopy in $E_\Sigma$; see e.g.~\cite[Section~11.3]{disk-embedding}.  Thus each double point gives rise to a monomial $\pm t^k \in \Lambda/\langle t^r \sim t^{-r}  :  r \in \Z\rangle$. To define the target, consider $\Lambda$ as a free abelian group with basis $\{t^r  :  r \in \Z\}$, and take the quotient abelian group where $t^r$ and $t^{-r}$ are identified for all~$r$.  Taking the sum of contributions from all the double points~$p$ gives
\[\mu(S) := \sum_p \varepsilon_p \gamma_p \in \Lambda/\langle t^r \sim t^{-r}  :  r \in \Z\rangle.\]
Let $e(S) \in \Z$ be the Euler number of the normal bundle of $S$, and let $\iota \colon \Z \to \Lambda$ be the unique such ring homomorphism. Then we have the useful formula
\begin{equation}\label{eqn-lambda-mu}
  \lambda([S],[S]) = \mu(S) + \overline{\mu(S)}  + \iota(e(S));
\end{equation}
see \cite[Theorem~5.2]{Wall-surgery-book}.
Here, we must fix a representative of $\mu(S)$ in $\Lambda$, and note  that the sum $\mu(S) + \overline{\mu(S)} \in \Lambda$ is independent of the choice of representative.
The next lemma will be useful for computing intersections numbers.

\begin{lem}\label{lem:computing-intersections}
  Let $\Sigma \subseteq S^4$ be a $\Z$-surface, and let $T \subseteq E_\Sigma$ be a rim torus. Let $D \subseteq E_{\Sigma}$ be a $\Z$-trivial immersed surface with one boundary component that is a longitude of $T$. Suppose that the interior of $D$ intersects $T$ transversely, and  suppose that a section of $\nu_{\partial D \subseteq T}$ extends to a nonvanishing section of $\nu_{D \subseteq E_{\Sigma}}$.
  Let $S$ denote the result of surgering $T$ along $D$.
  Let $G \subseteq E_{\Sigma}$ be an immersed $\Z$-trivial surface disjoint from $T$ and intersecting $D$ transversely in precisely one point, i.e.\  $|D \pitchfork G|=1$.
  \begin{enumerate}[font=\normalfont, label = (\roman*), ref= \roman*]
    \item\label{it:comp-intersections-i} Every intersection point of $\Int(D)$ with $T$ contributes $\pm t^k(1-t)$ to $\mu(S)$, for some sign and some $k \in \Z$.
    \item\label{it:comp-intersections-iii}
    Possibly after changing the orientation and basing arc of $G$, we have that $\lambda([S],[G]) = 1-t$.
  \end{enumerate}
\end{lem}

\begin{proof}
The surgery cuts $T$ open along $\partial D$ and glues in two parallel copies of $D$. Since $D$ is embedded, and we use the section of  $\nu_{D \subseteq E_{\Sigma}}$ from the statement, we obtain a surface with trivial normal bundle.

We first prove part~\eqref{it:comp-intersections-i}. Consider an intersection point $p \in \Int(D) \pitchfork T$. This gives rise to two double points $p_+$ and $p_-$ of $S$, which lie at either end of an arc $D^1 \subseteq T$ arising from the two parallel copies $D_+$ and $D_-$ of $D$.
Since $D_+$ and $D_-$ have opposite orientations induced on them from an orientation of $S$, the signs of the double points $p_+$ and $p_-$ differ. We can compare the corresponding group elements by taking a path $\gamma$ embedded in $D$, with $\gamma(0) = p$ and $\gamma(1) \in \partial D \subseteq T$. The path $\gamma$ gives rise to $\gamma \times D^1$, with $\gamma \times S^0 \subseteq D_+ \cup D_-$. Let $M$ be a meridional disc of $\Sigma$ with boundary on $T$, such that $\gamma(1) \times D^1$ is an arc of $\partial M$. Then $\delta := (\gamma \times D^1) \cup M \subseteq S^4$ is a disc with
\[
\partial \delta = (\gamma(0) \times D^1) \cup  (\gamma \times S^0)  \cup  \bigl(\partial M \setminus (\gamma(1) \times D^1) \bigr);
\]
see \cref{fig:computingints-1}.
The homotopy class of $\partial \delta$ in $\pi_1(E_\Sigma) \cong \Z$ measures the difference between the double point loops of the two double points. Since $\delta$ intersects $\Sigma$ exactly once, the difference in the contributions to $\mu(S)$ is $t^{\pm 1}$.  So, up to multiplication by~$\pm t^k$, we have a contribution $1-t$ or $1-t^{-1}$. But these are equal in $\Lambda/\langle t^r \sim t^{-r}  :  r \in \Z\rangle$, so this completes the proof of part~\eqref{it:comp-intersections-i}.

\begin{figure}
  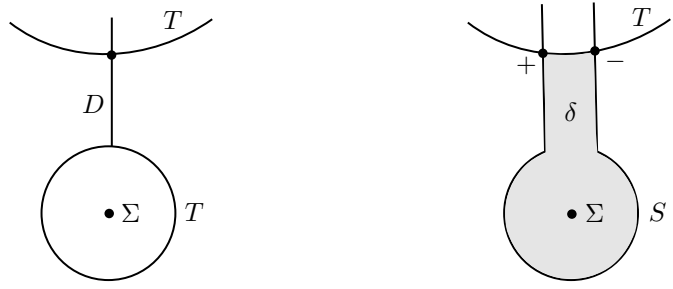
 \caption{Left: a schematic showing a meridian of the rim torus $T$, a portion of the surface $D$, and another sheet of $T$ intersecting $D$. Right: after surgery, with the disc $\delta$ that is used to compute the difference in the double point loops of the two intersection points in $\Int(D) \pitchfork T$ shown. }
 \label{fig:computingints-1}
  \end{figure}

  The proof of part~\eqref{it:comp-intersections-iii} is analogous to the proof of part~\eqref{it:comp-intersections-i}. The intersection point in $D \pitchfork G$ gives rise to two intersection points between $S$ and $G$. By the same argument as in the proof of part~\eqref{it:comp-intersections-i}, they contribute  $1-t$ or $1-t^{-1}$, up to multiplication by~$\pm t^k$. But $-t(1-t^{-1}) = 1-t$, so, up to units, there is again just one case.  By changing the basing arc of $G$ and its orientation, we can remove the up-to-units indeterminacy $\pm t^k$.
\end{proof}

The topological unknotting conjecture for genus $g \in \{1,2\}$ is implied by the following algebraic conjecture, which we advertise here.
We say that a hermitian, sesquilinear form on a free module over~$\Z[\Z]$ is \emph{surface hyperbolic} if it is isometric to $\cH_2^{\oplus g}$ for some $g$.
Such a form $(P,\lambda)$ is \emph{stably surface hyperbolic} if there is an isometry $(P,\lambda) \oplus \cH_2^{\oplus k} \cong \cH_2^{\oplus g}$ for some $k$, $g$.

\begin{conj}
Let $(P,\lambda)$ be a hermitian, sesquilinear form on a free $\Z[\Z]$-module of rank $2g$, where~$g \in \{1,2\}$, that is stably surface hyperbolic. Then $(P,\lambda)$ is surface hyperbolic.
\end{conj}

An important part of \cite{Conway-Powell} was to prove the analogue of this conjecture for $g \geq 3$.

\section{Turned twist-spun tori}\label{sec:twist-spun-tori}
	
	In this section, we review Boyle's construction \cite[Definition~5.1]{Boyle}.
	Let $K$ be a knot in $D^3$ such that $K$ intersects $S^2_{1/2}$ in $\{(0,0,1/2), (0,0,-1/2)\}$ transversely, and such that $K \setminus B^3_{1/2}$ is a boundary-parallel arc in $D^3 \setminus B^3_{1/2}$ with endpoints on $S^2_{1/2}$. Furthermore, let $i$ and $j$ be integers.
	
	Let $R_\theta$ be rotation of $\R^3$ about the $z$-axis by angle $\theta$. Fix $\varepsilon \in (0,1/2)$.
	Choose an isotopy $\{\Phi_t  :  t \in I\}$ of $D^3$ such that $\Phi_t|_{S^2_r} = \Id_{S^2_r}$ for all $r \in (1-\varepsilon, 1]$, and
	\[
	\Phi_t|_{D^3_{1/2}} = R_{2\pi j t} |_{D^3_{1/2}}
	\]
	for every $t \in I$. Write $\phi_j := \Phi_1$.
	
	We define the diffeomorphism $\eta_i \colon S^2 \times S^1 \to S^2 \times S^1$ via
	\[
	\eta_i|_{S^2 \times \{\theta\}} := R_{i \theta}|_{S^2} \times \Id_{\{\theta\}}
	\]
	for each $\theta \in S^1$. Consider the decomposition
	\[
	S^4 \cong \partial D^5 \cong \partial (D^3 \times D^2) \cong  (D^3 \times \partial D^2) \cup (\partial D^3 \times D^2).
	\]
	Then the 4-manifold
	\[
	(D^3 \times \partial D^2) \cup_{\eta_i} (\partial D^3 \times D^2)
	\]
	is obtained from $S^4$ via a Gluck twist along the standard 2-sphere $S_0 := \partial D^3 \times \{0\}$, and is hence diffeomorphic to $S^4$.  On the other hand, if $i$ is even, then $\eta_i$ is isotopic to $\Id_{S^2 \times S^1}$,  and so we again obtain the standard $S^4$.
	
	\begin{defn}\label{def:torus}
		The \emph{$i$-turned $j$-twist-spun torus of $K$} is the pair $(S^4, T_{K,i,j})$ given by
		\[
		\frac{(D^3, K) \times I}{\{(x,0) = (\phi_j(x), 1)  :  x \in D^3\}} \cup_{\eta_i} (S^2 \times D^2),
		\]
		where $S^4 \cong (D^3 \times \partial D^2) \cup_{\eta_i} (\partial D^3 \times D^2)$ and $T_{K,i,j} := K \times I$.
	\end{defn}
	
	
	Intuitively, we can describe $T_{K,i,j}$ as follows. Let $K \subseteq \R^3_{>0}$ be a knot. Choose a 3-disc $D \subseteq \R^3_{>0}$ such that $\partial D$ intersects $K$ transversely in a pair of antipodal points and $K \setminus D$ is an unknotted arc. Then we obtain $T_{K,i,j}$ in $\R^4 \subseteq S^4$ by rotating $\R^3_+$ around $\R^2 = \partial \R^3_+$, and while doing so, rotating $K \cap D$ around the line through $K \cap \partial D$ and simultaneously rotating $K$ around the $z$-axis.
	
	Boyle showed that the smooth isotopy class of $T_{K,i,j}$ only depends on $i$ through the parity of $i$. Indeed, the group of automorphisms of $S^2 \times S^1$ modulo ones that  extend to $S^2 \times D^2$ is $\Z/2$, generated by~$\eta_1$.
Boyle calls $T_{K,0,i}$ the \emph{$j$-twist-spun torus of $K$} and $T_{K,1,j}$ the \emph{turned $j$-twist-spun torus of $K$.}
	
	Boyle also showed that $\pi_1(S^4 \setminus T_{K,i,j}) \cong G / [G, \mu^j]$, where $G \cong \pi_1(S^3 \setminus K)$ and $\mu \in G$ is the class of a meridian of $K$. Since $\mu$ normally generates $G$, we have $[G, \mu] \cong [G, G]$. This isomorphism can be seen using the commutator identities $[g,\mu^x] = \bigl[g^{x^{-1}},\mu \bigr]^x$ and $[g,\mu^x\mu^y]= [g,\mu^x][g,\mu^y]^{\mu^x}$, where superscripts indicate conjugation. Hence
	\[
	\pi_1(S^4 \setminus T_{K,i,1}) \cong G / [G, G] \cong H_1(S^3 \setminus K) \cong \Z.
	\]
	In fact, $T_{K,0,1}$ is smoothly isotopic to $T^2$, since it is the 1-twist-spun 2-knot obtained from $K$, which is the unknot by Zeeman~\cite{Zeeman-twisting}, with a trivial tube attached. However, Boyle could not determine whether $T_{K,1,1}$ was smoothly equivalent to $T^2$, and this question is still open. Note that one can obtain~$T_{K,1,1}$ from~$T_{K,0,1}$ via a Gluck twist along the trivial 2-knot $S_0$ that links $T_{K,0,1}$ nontrivially.
	
	\section{Every turned 1-twist-spun torus is topologically unknotted}\label{section:proof-turned-tori-unknotted}
	
	In this section, we study turned 1-twist-spun tori up to topological isotopy. Our main result is the following, restated from the introduction.
	
	\Boyle*
	
	\begin{proof}
		By \cref{lemma:main-strategy}, it suffices to find $\Z$-trivial surfaces in
\[E := E_{T_{K,1,1}} = S^4 \setminus N(T_{K,1,1}),\]
with respect to which the $\Lambda$-intersection form $\lambda$ of $E$ is represented by a matrix of the form $A$ as in \cref{lemma:main-strategy}.
		Throughout, we use the notation of Definition~\ref{def:torus}.
		Consider the sphere
		\[
		S_0 := \partial D^3 \times \{0\} \subseteq S^4
		\]
		disjoint from $K$. This has self-intersection $0$.
		
		Isotope the arc $K \setminus B^3_{1/2}$ relative to its boundary such that its interior only intersects the $z$-axis in $\{(0,0,3/4)\}$, where it is perpendicular to it. Let $\mu$ be a meridian of $K$ at $(0,0,3/4)$ of radius $1/8$.  Then
		\[
		T_{\ell} := \frac{\mu \times I}{\{(x,0) = (\phi_1(x), 1)  :  x \in \mu\}} \subseteq E
		\]
		is a rim torus of $T_{K,1,1}$ over a longitude $\ell$ of $T_{K,1,1}$.
		
		Since $T_{\ell}$ is not a $\Z$-trivial surface, we surger it to a sphere. For this end, consider the arc
		\[
		J := \{0,0\} \times [7/8, 1] \subseteq D^3.
		\]
		Then
		\begin{equation} \label{eqn:disc}
		\begin{split}
			D &:= \frac{J \times I}{\{(x,0) = (\phi_1(x), 1)  :  x \in J\}} \cup_{\eta_1} (\{(0,0,1)\} \times D^2) \\
			&\subseteq \frac{D^3 \times I}{\{(x,0) = (\phi_1(x), 1)  : x \in D^3]\}} \cup_{\eta_1} (S^2 \times D^2)
		\end{split}
		\end{equation}
		is a disc in $S^4$ with boundary on $T_{\ell}$.
		
		Since $\eta_1$ applies a $2\pi$ rotation about the $z$-axis to $\mu$, a trivialisation of the normal bundle $\nu_{\partial D \subseteq T_{\ell}}$ does not extend to a section of the normal bundle $\nu_{D \subseteq E}$. The tangential framing of $\nu_{\partial D \subseteq T_{\ell}}$, transported to
\[ \frac{\partial D^3 \times I}{\{(x,0) = (\phi_1(x), 1)  :  x \in \partial D^3\}} =  \frac{\partial D^3 \times I}{\{(x,0) = (x, 1)  :  x \in \partial D^3\}} \cong S^2 \times S^1,\]
is a constant vector \[\underline{v} \times \{\underline{0}\} \in T_{(0,0,1)}S^2 \times T_{\theta}S^1 \subseteq TS^2 \times TS^1 \cong T(S^2 \times S^1).\]
A nonvanishing section of the normal bundle of $\{(0,0,1)\} \times D^2$ in $S^2 \times D^2$ also gives rise to a constant vector, in the coordinates of $S^2 \times D^2$. However, the identification $\eta_1$ is a Gluck twist, so introduces a full rotation of difference between the two. Hence $D$ is not compatibly framed with respect to the tangential framing of $\partial D$.

		However, if we perform a boundary twist on $D$ along the curve $\partial D = \{(0,0,7/8)\} \times I$, then we can undo the twist. The resulting disc $D_1$ intersects $T_{\ell}$ once, but now a trivialisation of $\nu_{\partial D_1 \subseteq T_{\ell}}$ does extend to $\nu_{D_1 \subseteq E}$.
		
		If we surger $T_{\ell}$ along $D_1$, we obtain an immersed sphere $S_1$ with two double points, and we claim that $\mu([S_1]) = \pm(1-t) \in \Lambda/\langle t^r \sim t^{-r}  :  r \in \Z\rangle$.
We have that $D_1$ is embedded and framed, and intersects $T$ in one point, so, by part~\eqref{it:comp-intersections-i} of \cref{lem:computing-intersections}, we have that $\mu(S_1) = \pm t^k (1-t)$ for some $k$.
Since the double points arose from a boundary twist, it follows that one of the double points is local, i.e.\ the entire double point loop is contained in a copy of $D^4 \subseteq E_\Sigma$. Hence, this double point loop is trivial. We deduce that $\mu(S_1) = \pm(1-t)$ or $\pm (1-t^{-1})$. Given the indeterminacy in $\mu(S_1)$, this completes the proof of the claim that $\mu(S_1) =\pm(1-t)$.

 Furthermore, thanks to the boundary twist, $S_1$ has trivial normal bundle, hence $e(S_1)=0$ and so, by equation~\eqref{eqn-lambda-mu}, we have that
		\[
		\lambda([S_1], [S_1]) = \mu([S_1]) + \overline{\mu([S_1])} = \pm (1-t) \pm (1-t^{-1}).
		\]
		As $|S_0 \pitchfork D_1| = 1$, by part~\eqref{it:comp-intersections-iii} of \cref{lem:computing-intersections}, we have that $\lambda([S_0],[S_1]) = 1-t$, after possibly changing the orientation of $S_0$ and its basing arc.
Since $\lambda$ is Hermitian, $\lambda([S_1],[S_0]) = 1-t^{-1}$.
Hence, the matrix of $\lambda$ on the submodule generated by the classes $[S_0]$ and $[S_1]$ is
		\[
		\begin{pmatrix}
			0 & 1-t \\
			1-t^{-1} & \pm (1-t) \pm (1-t^{-1})
		\end{pmatrix}.
		\]
		It follows from \cref{lemma:main-strategy}, with $g=1$,
that $([S_0], [S_1])$ is indeed a basis of $H_2(E; \Lambda) \cong \Lambda^2$, that the equivariant intersection form $\lambda$ is isometric to $\cH_2$, and that $\Sigma$ is topologically unknotted, as desired.
	\end{proof}

\begin{rem}
	An alternative strategy to proving Theorem~\ref{thm:Boyle} is to observe that the 1-twist-spun torus $T_{K,1,1}$ can be obtained from the standard torus $T^2$ in $S^4$ by performing 1-twist rim surgery along the curve with slope~1, and then applying \cite[Theorem~1.7]{Conway-Powell}.
\end{rem}

\section{Knot surgery}
\label{sec:knot-surgery}

In this section, we extend the results of Kim and Ruberman~\cite{Kim-Ruberman} to $\Z$-surfaces in $S^4$.
Let $\Sigma \subseteq S^4$ be a smoothly embedded, topologically unknotted, oriented surface of genus $g$. Furthermore, let $T \subseteq S^4 \setminus N(\Sigma)$ be a smoothly embedded torus that is unknotted in $S^4$.
If $J$ is a knot in $S^3$ with exterior $E_J$ and $t$, $r \in \Z$, then we let
\[
S^4_{t, r}(T, J) := (S^4 \setminus N(T)) \cup_f (E_J \times S^1)
\]
be the result of \emph{$t$-twist $r$-roll knot surgery along $T$ with pattern $J$}, where we describe the glueing map $f$ next.

Fix a parametrisation $g \colon S^1 \times S^1 \to T$, and write $\ell := g(S^1 \times \{1\})$ and $m := g(\{1\} \times S^1)$. Let $s \colon T \to \partial \overline{N(T)}$ be a section such that $s(\ell)$ and $s(m)$ are trivial in $H_1(S^4 \setminus N(T)) \cong \Z$. If $\mu_T \subseteq \partial \overline{N(T)}$ is a meridian of $T$, then
\[
\bigl\{[s(\ell)], [s(m)], [\mu_T]\bigr\}
\]
is a basis of $H_1 \bigl(\partial \overline{N(T)}\bigr) \cong H_1(T^3) \cong \Z^3$.

Let $\mu_J \subseteq \partial E_J$ be a meridian and let $\lambda_J \subseteq \partial E_J$ be a 0-framed longitude of $J$, and write $\mu_J \cap \lambda_J = \{b\}$. Then
\[
\bigl\{[\{b\} \times S^1], [\mu_J \times \{1\}], [\lambda_J \times \{1\}]\bigr\}
\]
is a basis of $H_1(\partial(E_J \times S^1)) = H_1(\partial E_J \times S^1) \cong H_1(T^3) \cong \Z^3$.

Note that $\pi_0 \bigl(\text{Diff}^+(T^3) \bigr) \cong \text{SL}(3, \Z)$, with the isomorphism given by taking the induced map on $H_1(T^3)$, by Waldhausen~\cite{Waldhausen}. Hence, we can define the map $f \colon \partial \overline{N(T)} \to \partial E_J \times S^1$ up to isotopy by requiring that
\[
f(s(\ell)) = \{b\} \times S^1 + t (\mu_J \times \{1\}) + r (\lambda_J \times \{1\}) \text{, } f(s(m)) = \mu_J \times \{1\} \text{, and } f(\mu_T) = \lambda_J \times \{1\}.
\]
Since $T$ is smoothly unknotted in $S^4$, there is a diffeomorphism
\[
\Theta \colon S^4_{t, r}(T, J) \to S^4;
\]
see Kim--Ruberman~\cite{Kim-Ruberman}. We write
\[
\Sigma_{t, r}(T, J) := \Theta(\Sigma).
\]
When $T$ is a rim torus of $\Sigma$ and $m$ is a meridian of $\Sigma$, then this is known as the $t$-twist $r$-roll rim surgery of $\Sigma$ with pattern $J$, but in our application $T$ will not be a rim torus. We now restate our main result on knot surgeries in surface complements from the introduction.

\knotsurgery*

\begin{proof}
Our discussion below works for arbitrary $g$, but it follows from the work of Conway and the second author~\cite{Conway-Powell} that $\Sigma_{t, r}(T, J)$ is topologically standard if it is a $\Z$-surface and $g \not\in \{1, 2\}$. The strategy is as follows. 	Let $E :=S^4 \setminus N(\Sigma_{t,r}(T, J))$ and $E_\Sigma := S^4 \setminus N(\Sigma)$.  We describe generators for $H_2(E_{\Sigma};\Lambda)$, and then modify them during the knot surgery construction to obtain generators for $H_2(E;\Lambda)$. We check that the intersection form is unchanged.

Now we begin to implement this. Since $\Sigma$ is topologically unknotted, there is a locally flat 3-manifold $Y \subseteq S^4$ homeomorphic to $S^3$ such that $\Sigma \subseteq Y$ and the closures $H_\alpha$ and $H_\beta$ of the components of $Y \setminus \Sigma$ are handlebodies.
Let~$\alphas = \{\alpha_1, \dots, \alpha_g\}$ and $\betas = \{\beta_1, \dots, \beta_g\}$ be collections of pairwise disjoint simple closed curves on~$\Sigma$ such that $\alphas \cup \betas$ represents a basis of~$H_1(\Sigma)$, and $\alpha_i$ bounds a locally flat disc $D_{\alpha_i} \subseteq H_\alpha$ and $\beta_i$ bounds a locally flat disc $D_{\beta_i} \subseteq H_\beta$ for $i \in \{1, \dots, g\}$. Furthermore, we can assume that $|\alpha_i \pitchfork \beta_j| = \delta_{ij}$ and the interiors of the discs $D_\gamma$ for $\gamma \in \alphas \cup \betas$ are pairwise disjoint.

Let $R_\gamma \subseteq S^4 \setminus \Sigma$ be a rim torus around $\gamma$ for $\gamma \in \alphas \cup \betas$, which is the boundary of the restriction $N_\gamma$ of the unit normal disc bundle of $\Sigma$ to $\gamma$. Let $S_\gamma \subseteq S^4 \setminus \Sigma$ be the topologically embedded 2-sphere obtained by surgering $R_\gamma$ along the topologically embedded disc $D_\gamma' := D_\gamma \setminus \Int(N_\gamma)$.
We can arrange that $D_{\alpha_i}' \cap R_\gamma = \emptyset$ and $D_{\beta_i}' \cap R_\gamma = \emptyset$ for any $i \in \{1, \dots, g\}$ and $\gamma \in (\alphas \cup \betas) \setminus \{\alpha_i, \beta_i\}$.
It follows that \begin{equation}\label{eqn:disjoint-surfaces-Ss} S_{\alpha_i} \cap S_{\alpha_j} = S_{\beta_i} \cap S_{\beta_j} = S_{\alpha_i} \cap S_{\beta_j}= \emptyset\end{equation}
for every $i$, $j \in \{1, \dots, g\}$ with  $i \neq j$. All the $S_{\gamma}$ are topologically embedded spheres in $S^4$, so have trivial normal bundle. Thus their self-intersection numbers all vanish, i.e.\ $\lambda_{E_{\Sigma}}([S_\gamma],[S_\gamma])=0$.
The collection $\{[S_{\alpha_1}], [S_{\beta_1}], \dots, [S_{\alpha_g}], [S_{\beta_g}] \}$ is a basis of $H_2(S^4 \setminus N(\Sigma); \Lambda) \cong \Lambda^{2g}$, and in this basis the $\Lambda$-intersection form $\lambda_{S^4 \setminus N(\Sigma)} \cong \bigoplus_g \cH_2$; see Section~\ref{section:std-int-form-implies-unknotted}.

\begin{lem}\label{lem:pi1}
	Suppose that $s(m)$ is 0-homotopic in $S^4 \setminus (\Sigma \cup T)$. Then $\Sigma_{t,r}(T, J)$ is a $\Z$-surface.
\end{lem}

\begin{proof}
	Let $E :=S^4 \setminus N(\Sigma_{t,r}(T, J))$ and $E_\Sigma := S^4 \setminus N(\Sigma)$. Then
	\[
	E =(E_\Sigma \setminus N(T)) \cup_{\partial \overline{N(T)}} (E_J \times S^1).
	\]
	Hence, by the Seifert--van Kampen theorem,
	\[
	\pi_1(E) \cong \pi_1(E_\Sigma \setminus N(T)) \ast_{\Z^3} (\pi_1(E_J) \times \Z),
	\]
	where $\pi_1(T^3) \cong \Z^3$ is generated by the curves $\mu_J \times \{1\}$, $\lambda_J \times \{1\}$, and $\{b\} \times S^1$ (recall that $\{b\} := \mu_J \cap \lambda_J$).
	Since $\mu_J \times \{1\} = f(s(m))$ and $s(m)$ is 0-homotopic in $E_\Sigma \setminus N(T)$ by assumption, $\mu_J \times \{1\}$ is 0-homotopic in $E$. As $\mu_J$ normally generates $\pi_1(E_J)$, the curve $\lambda_J \times \{1\}$ is also 0-homotopic in $E$. However, $f(\mu_T) = \lambda_J \times \{1\}$, so $\mu_T$ is 0-homotopic in $E$. Finally, as $f(s(\ell)) = \{b\} \times S^1 + t (\mu_J \times \{1\}) + r (\lambda_J \times \{1\})$ and $\mu_J \times \{1\}$ and $\lambda_J \times \{1\}$ are 0-homotopic in $E$, the curve $\{b\} \times S^1$ is homotopic to $s(\ell)$ in $E$. Hence
	\[
	\pi_1(E) \cong \pi_1(E_\Sigma \setminus N(T))/\langle \mu_T \rangle \cong \pi_1(E_\Sigma) \cong \Z,
	\]
	where $\langle \mu_T \rangle$ is the normal subgroup generated by $\mu_T$, and the second isomorphism follows from the decomposition $E_\Sigma = (E_\Sigma \setminus N(T)) \cup N(T)$ via the Seifert--van Kampen theorem.
\end{proof}

The intersection $S_\gamma \cap N(T)$ is a collection of meridional discs. The boundary of each disc is a meridian of $T$. In the definition of the glueing map $f \colon \partial \overline{N(T)} \to \partial E_J \times S^1$, we have that $f(\mu_T) = \lambda_J \times \{1\}$. Hence each meridian of $T$ is taken to a parallel of $J$. 
For each $\gamma \in \alphas \cup \betas$, construct $S_\gamma' \subseteq S^4_{t, r}(T, J)$ by first making $S_\gamma$ transverse to $T$, and replacing each component of $S_\gamma \cap N(T)$, which is a disc, with a Seifert surface of $J$ in $E_J \times \{p\}$ for some $p \in S^1$. We can assume that we have a different $p$ for every component, so these Seifert surfaces are mutually disjoint.

We have assumed that $s(m)$ is 0-homotopic in $S^4 \setminus (\Sigma \cup T)$. So, by the proof of Lemma~\ref{lem:pi1}, the map $\pi_1(E_J) \to \pi_1(S^4_{t, r}(T, J) \setminus N(\Sigma))$ induced by the embedding is trivial. Hence $S_\gamma'$ is a $\Z$-surface in $S^4_{t,r}(T, J) \setminus N(\Sigma)$.
We have that
\begin{equation}\label{eqn:disjoint-surfaces-S-primes} S_{\alpha_i}' \cap S_{\alpha_j}' = S_{\beta_i}' \cap S_{\beta_j}' = S_{\alpha_i}' \cap S_{\beta_j}'= \emptyset\end{equation}
for every $i$, $j \in \{1, \dots, g\}$ with $i \neq j$. This follows from the analogous facts in \eqref{eqn:disjoint-surfaces-Ss} for the $S_{\alpha_i}$ and the~$S_{\beta_i}$, together with the fact that the Seifert surfaces used in the construction of the $S_{\gamma}'$ are mutually disjoint.

As above let \[E := S^4 \setminus N(\Sigma_{t,r}(T, J)).\]
We can arrange that $S_{\alpha_i} \cap S_{\beta_j} \subseteq S^4 \setminus N(\Sigma \cup T)$ for every $i$, $j \in \{1, \dots, g\}$. Then $S_{\alpha_i}' \cap S_{\beta_j}' \subseteq S^4 \setminus N(\Sigma \cup T)$ for every $i$, $j \in \{1, \dots, g\}$ as well. Let
\[
U_\gamma := \Theta(S_\gamma') \subseteq E
\]
for $\gamma \in \alphas \cup \betas$, oriented in such a way that $\Theta|_{S_\gamma'} \colon S_\gamma' \to U_\gamma$ is orientation-preserving.
We compute the intersection numbers among the $U_{\gamma}$.
Since each $S_{\gamma}'$ is an embedded $\Z$-surface, each $U_{\gamma}$ is an embedded $\Z$-surface. Since $U_{\gamma} \subseteq S^4$, it has trivial normal bundle.
 Thus $\lambda_E([U_\gamma],[U_{\gamma}])=0$ for every $\gamma \in \alphas \cup \betas$.
 Since $\Theta$ is a diffeomorphism, the analogous facts in \eqref{eqn:disjoint-surfaces-S-primes} for the $S_{\gamma}'$  imply that
 \[U_{\alpha_i} \cap U_{\alpha_j} = U_{\beta_i} \cap U_{\beta_j} = U_{\alpha_i} \cap U_{\beta_j}= \emptyset\]
 for every $i$, $j \in \{1, \dots, g\}$ with $i \neq j$.
In order to complete the computation of the intersection numbers among the $U_\gamma$, it remains to compute $\lambda([U_{\alpha_i}],[U_{\beta_i}])$, for $i=1,\dots,g$.

\begin{prop}\label{prop:intersection}
	Let $q \in S_{\alpha_i} \cap S_{\beta_i}$ be an intersection point with sign $\varepsilon \in \{+, -\}$ and double point group element $t^k \in \langle t \rangle \cong C_\infty \cong \Z$. Since $q \in S^4 \setminus N(\Sigma \cup T)$, we can consider $\Theta(q) \in U_{\alpha_i} \cap U_{\beta_i} \subseteq S^4 \setminus N(\Sigma_{t, r}(T, J))$. Then $\Theta(q)$ also has sign $\varepsilon$ and double point group element $t^k$. It follows that $\lambda_E([U_{\alpha_i}],[U_{\beta_i}]) = \lambda_{E_{\Sigma}}([S_{\alpha_i}],[S_{\beta_i}]) = 1-t$.
\end{prop}

\begin{proof}
	As $\Theta$ is orientation-preserving and the orientations of $U_{\alpha_i}$ and $U_{\beta_i}$ are induced by $\Theta$, the intersection point $\Theta(q)$ has the same sign as $q$.
	
	Let $V \subseteq S^4$ be a disc such that $\partial V$ is a double point loop for $q$. Then $V$ intersects $\Sigma$ transversely in~$k$ points, counted with sign. We can arrange that $V$ is transverse to $T$. Replace $V \cap N(T)$ with Seifert surfaces of $J$ to obtain $V' \subseteq S^4_{t, r}(T, J)$, and write $W := \Theta(V') \subseteq S^4$. Then $\partial W$ is a double point loop for $\Theta(q)$, and $W$ intersects $\Sigma_{t, r}(T, J)$ transversely in $k$ points, again counted with sign. As
$\Theta \colon S^4_{t,r}(T, J) \setminus N(\Sigma) \to S^4 \setminus N(\Sigma_{t,r}(T, J))$ induces an isomorphism on fundamental groups, both of which are isomorphic to $\Z$, it follows that the double point group element of $\Theta(q)$ is $t^k$.  We deduce that $\lambda_E([U_{\alpha_i}],[U_{\beta_i}]) = \lambda_{E_{\Sigma}}([S_{\alpha_i}],[S_{\beta_i}])$ for $i \in \{1,\dots,g\}$.
\end{proof}

By Proposition~\ref{prop:intersection} and the computations above it, the matrix of $\lambda_E$ on $\{U_{\alpha_1}, U_{\beta_1}, \dots, U_{\alpha_g}, U_{\beta_g}\}$ is $\bigoplus_g \cH_2$.
It follows from Lemma~\ref{lemma:main-strategy} that $\{U_{\alpha_1}, U_{\beta_1}, \dots, U_{\alpha_g}, U_{\beta_g}\}$ is a basis of $H_2(E; \Lambda)$, and that
\[
\lambda_E \cong \bigoplus_g \cH_2.
\]
Thus $\Sigma_{t, r}(T, J)$ is topologically unknotted by Theorem~\ref{thm:intersection-form}. This concludes the proof of Theorem~\ref{thm:knot-surgery}.
\end{proof}

\section{The Cochran--Davis Seifert surface union a ribbon disc}\label{sec:Cochran-Davis-examples}

A natural idea for constructing potentially nontrivial, closed, orientable $\Z$-surfaces, smoothly embedded in $S^4$, is to take the union of a Seifert surface $F$ for a slice knot $K$ in $S^3$ and a homotopy ribbon disc $D$ for $K$. However, by the following lemma, this is smoothly unknotted if $D$ is obtained from $F$ by compressing along a slice derivative, which is often the case. Recall that a slice derivative for a genus $g$ Seifert surface $F$ for a knot $K$ in $S^3$ is a collection of $g$ pairwise disjoint, homologically linearly independent simple closed curves that form a smoothly slice link such that the Seifert form vanishes identically on the submodule of $H_1(F)$ generated by these curves.

\begin{lem}\label{lem:slice-derivative}
	Let $F \subseteq S^3$ be a genus $g$ Seifert surface, and suppose that $c \subseteq F$ is a slice derivative. Write $D$ for the slice disc obtained by surgering $F$ along $c$. Then $\Sigma := F \cup D$ is smoothly unknotted.
\end{lem}

\begin{proof}
	Since $D$ is obtained by surgering $F$ along $c$, it follows that $\Sigma$ bounds an embedded copy of $F \times I$ with a collection of 3-dimensional 2-handles attached along $c \times \{0\}$, which we denote by $H$. Since the components of $c$ are $g$ homologically linearly independent simple closed curves in $F$, the 3-manifold $H$ is a genus $g$ handlebody. To see this, note that, if we turn the relative handle decomposition of $H$ built on $F \times I$ upside down, we obtain a handle decomposition of $H$ with only 1-handles relative to $D^2 \times I$, as $D^2$ is the surface obtained by compressing $F$ along $c$. So one can obtain $H$ by attaching 3-dimensional 1-handles to $D^2 \times I$, which is a handlebody. Since $\partial H = \Sigma$, it follows that $\Sigma$ is smoothly unknotted. 	
\end{proof}
	
Kauffman conjectured that every genus one Seifert surface for a slice knot admits a slice derivative, but Cochran and Davis~\cite{Cochran-Davis} constructed a counterexample. This motivates \cref{thm:ribbon}, which we restate for the reader's convenience.
	
\ribbon*

Before beginning the proof,  we recall the Cochran--Davis construction~\cite{Cochran-Davis}.
It starts with the ribbon knot $R$ with Seifert surface $F'$ on the left of \cref{fig:CD-knot-1}. We obtain a ribbon disc $\Delta \subseteq D^4$ for $R$ by cutting the left band with a saddle move, as shown in \cref{fig:CD-knot-2}, and then capping off the resulting 2-component unlink with discs. This is equivalent to surgering $F'$ along the slice derivative $\alpha'$, since cutting the left band of $F'$ results in an annulus with core $\alpha'$.

\begin{figure}
	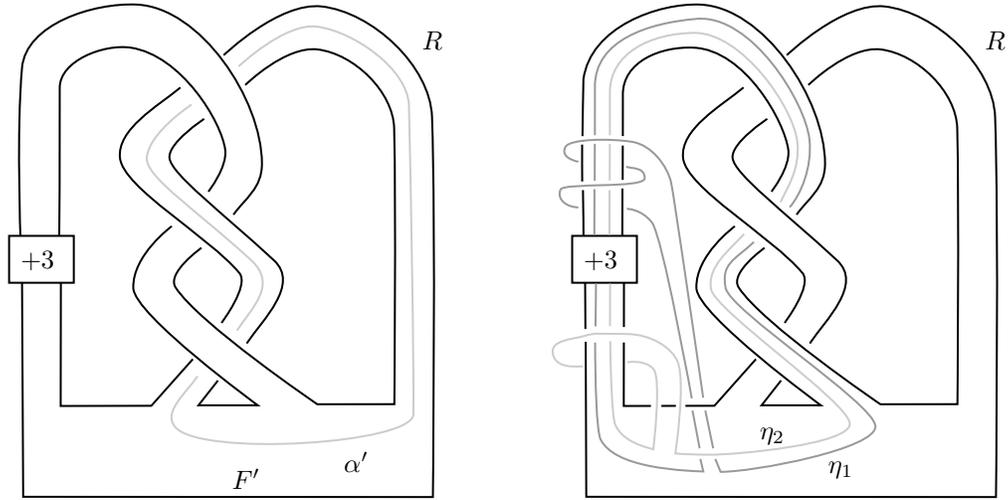
 \caption{Left: a ribbon knot $R$, together with a Seifert surface $F'$ and a slice derivative $\alpha'$ on $F'$.
  Right: the same knot $R$, with two embedded curves $\eta_1$ and $\eta_2$ in $S^3 \sm R$ such that $\eta_1 \sqcup \eta_2$ is an unlink. See \cite[Figure~6]{Cochran-Davis}.}
 \label{fig:CD-knot-1}
  \end{figure}

 The curves $\eta_1 \sqcup \eta_2$ shown on the right of \cref{fig:CD-knot-1} form an unlink, and cobound an annulus $A$ smoothly embedded in $D^4 \sm \Delta$ by~\cite[Proposition~5.1]{Cochran-Davis}. To construct this annulus, perform a saddle move on $\eta_1 \sqcup \eta_2$ and then on $R$, as shown in \cref{fig:CD-knot-2}. \cref{fig:CD-knot-2} now shows a 3-component unlink,
 so the result of the saddle moves can be capped off with disjoint discs to produce the rest of~$\Delta$ and the rest of~$A$.

\begin{figure}
	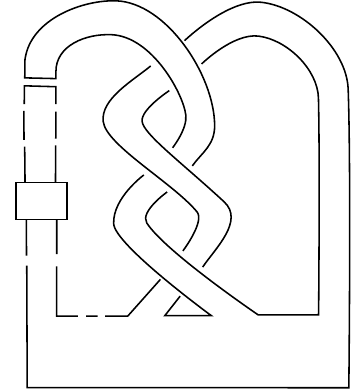
 \caption{The saddle moves for constructing the ribbon disc $\Delta$ and the annulus $A$.}
 \label{fig:CD-knot-2}
  \end{figure}

Cochran and Davis perform a satellite construction on $D^4$ along the annulus $A \subseteq D^4$ with $\partial A = \eta_1 \sqcup \eta_2 \subseteq S^3$.
We write
\[S^4 \cong S^4_+ \cup_{S^3} S^4_-,\]
where $S^4_{\pm} \cong D^4$.
Using the identification $S^4_+ \cong D^4$, we have $\Delta \subseteq S^4_+$. We push the Seifert surface $F'$ for $R$ into~$S^4_-$.
Consider the surface $\Sigma':= \Delta \cup F' \subseteq S^4_+ \cup S^4_-$. Relabel $A$ as $A_+ \subseteq S^4_+$, and let $A_-$ denote another copy of $A$ in $S^4_-$. By only pushing $F'$ into $S^4_-$ a small amount, so it lives in a collar
of the boundary $S^3$,  we may assume that $A_- \cap F' = \emptyset$.

Our goal is to perform the satellite construction mentioned above (and explained in detail below)
simultaneously on $A_+$ and~$A_-$, which will produce the knot $K$, with a ribbon disc $D$ and a Seifert surface $F$. Then $\Sigma := D \cup F \subseteq S^4$ is the $\Z$-torus that is the subject of \cref{thm:ribbon}.

Let $J$ be one of the knots in $S^3$ considered by Cochran and Davis, for example the left-handed trefoil, and let $E_J$ be its exterior.
Cochran and Davis parametrise $A$ as $S^1 \times [0,1]$, and then perform a 1-parameter satellite construction on $D^4$, removing a neighbourhood $S^1 \times [0,1] \times B^2 \cong S^1 \times B^2 \times [0,1]$, and glueing in $E_J \times [0,1]$ instead, in such a way that each meridian $\{p\} \times S^1 \times \{t\}$ is identified with a longitude of $J$ in $E_J \times \{t\}$ for $p \in S^1$ and $t \in I$. They show that the result of this construction is a manifold $W$ that comes with a diffeomorphism  $\Phi \colon W \to D^4$; see~\cite[Theorem~3.1]{Cochran-Davis}.
They define the knot $K$ to be $\Phi(R)$, the disc $D$ to be $\Phi(\Delta)$, and the Seifert surface $F$ to be $\Phi(F')$.

The knot $K$ can be obtained  with a purely 3-dimensional construction, as follows. Perform two simultaneous infections (satellite constructions) on the curves $\eta_1$ and $\eta_2$, with companion knots~$J$ and~$-J$ respectively.  The resulting knot $K:= R_{\eta_1,\eta_2}(J,-J)$ is the \emph{Cochran--Davis knot}.
	
\begin{proof}[Proof of \cref{thm:ribbon}]
We consider the decomposition $S^4 = S^4_+ \cup_{S^3} S^4_-$.
We perform the satellite construction on both $S^4_+$ and $S^4_-$ along the annuli $A_+$ and $A_-$, respectively. This produces 4-manifolds~$W_{\pm}$ with $\partial W_{\pm} \cong S^3$ and diffeomorphisms $\Phi_{\pm} \colon W_{\pm} \to S^4_{\pm}$.
The knot~$R$, the pushed-in Seifert surface~$F'$, and the disc~$\Delta$ are all disjoint from $A_{\pm}$, and so represent analogous objects in $W_{\pm}$. We then have the pushed-in Seifert surface $F := \Phi_-(F') \subseteq S^4_-$ for the knot $K := \Phi_{\pm}(R) \subseteq S^3$, and ribbon disc $D := \Phi_+(\Delta) \subseteq S^4_+$. So
\[
\Sigma := F \cup D = \Phi_-(F') \cup \Phi_+(\Delta).
\]
Hence $\Sigma$ is obtained from $\Sigma' = F' \cup \Delta$ via knot surgery along the torus $T_A := A_+ \cup A_-$ with pattern the knot $J$.
To see this, note that knot surgery removes $N(T)$ and glues in $E_J \times S^1$. Our construction removes $N(A_+) \cup N(A_-)$ and glues in $(E_J \times [0,1]) \cup (E_J \times [0,1]) \cong E_J \times S^1$. As the glueing maps coincide, $\Sigma$ is indeed obtained via knot surgery.

Note that $\pi_1(D^4 \setminus F') \cong \Z$, generated by a meridian by~\cite[Lemma~3.1]{Chen}. Since $\Delta$ is homotopy ribbon, it follows from the Seifert--van Kampen theorem that $\Sigma'$ is a $\Z$-torus.

As $\eta_1 \sqcup \eta_2$ is an unlink, and we obtain $A$ from it by attaching a trivial band and capping off the resulting unknot with a disc, it follows that $A$ and hence $T_A$ are smoothly unknotted. The surface $\Sigma'$ is also unknotted by Lemma~\ref{lem:slice-derivative}, since one can obtain $\Delta$ by compressing $F'$ along the slice derivative $\alpha'$ shown on the left of Figure~\ref{fig:CD-knot-1}.

The annulus $A_-$ is obtained by pushing $A$ into $S^4_-$ deeper than $F_-$, hence the meridian $m$ of $A_-$ bounds a disc
whose interior is disjoint from $A_- \cup F_-$.
So the meridian of $T_A$ also bounds a disc in~$S^4 \setminus \Sigma'$ whose interior lies in $S^4 \setminus (\Sigma' \cup T_A)$.

So $\Sigma'$ and $T_A$ satisfy all the assumptions of Theorem~\ref{thm:knot-surgery}, and hence the result $\Sigma$ of knot surgery on~$\Sigma'$ along~$T_A$ with pattern $J$ is topologically unknotted, as claimed.
\end{proof}
		
We close this section by noting another general case when a genus one Seifert surface union a slice disc is a topologically unknotted $\Z$-torus, suggested by Anthony Conway.  Let $K \subseteq S^3 \subseteq S^4 = S^4_+ \cup S^4_-$ be a genus one knot with Alexander polynomial $\Delta_K \doteq 1$.  Let $D \subseteq S^4_+$ be a $\Z$-slice disc  for $K$, and let $F$ be a genus one Seifert surface for $K$ pushed into $S^4_-$.  Let $\Sigma := F \cup_K D$.

\begin{prop}\label{prop:alex-poly-one-torus}
	The torus $\Sigma \subseteq S^4$ is topologically unknotted.
\end{prop}

\begin{proof}
Let $E_D:= S^4_+ \sm N(D)$, $E_K:= S^3 \sm N(K)$, and $E_F := S^4_- \sm N(F)$. Then $E_\Sigma := S^4 \sm N(\Sigma) = E_D \cup_{E_K} E_F$ and $\pi_1(E_\Sigma) \cong \Z$.
With $\Lambda$-coefficients, $H_i(E_K;\Lambda) = 0$ and $H_i(E_D;\Lambda)=0$ for $i \in \{1,2\}$. The Mayer--Vietoris sequence implies that the inclusion-induced map $H_2(E_F;\Lambda) \xrightarrow{\cong} H_2(E_\Sigma;\Lambda)$ is an isomorphism.  By \cite[Proposition~7.10]{Conway-Powell}, the intersection form of $E_F$ is isometric to $\cH_2$.  It follows that the intersection form of $E_\Sigma$ is also isometric to $\cH_2$, and hence $\Sigma$ is topologically unknotted by \cref{thm:intersection-form}.
\end{proof}

\section{Tori with four critical points}\label{section:four-crit-points}

We now prove our theorem on tori with four critical points, which we recall here.

\morse*

\begin{proof}
  Denote the critical points of $\Sigma$ by $c_0$, $c_1^a$, $c_1^b$, and $c_2$, where the subscript denotes the index.  As before, write $S^4 = S^4_+ \cup_{S^3} S^4_-$, a union of two copies of $D^4$.
  Arrange the critical points of $\Sigma$ such that $c_0$ and $c_1^a$ lie in $S^4_- \cong D^4$, and $c_1^b$ and $c_2$ lie in $S^4_+ \cong D^4$.  The intersection of $\Sigma$ with $S^4_+ \cap S^4_- =S^3$ is a two-component link $L = L_1 \sqcup L_2$, with the property that a single fusion band move yields the unknot, in two different ways.

  Let $E_L := S^3 \sm N(L)$, and define \[W_{\pm} := S^4_{\pm} \sm N(\Sigma).\]
 As before, set $\Lambda := \Z[\Z]$.
Fix an identification $\pi_1(S^4 \sm N(\Sigma)) \cong \Z$.  We then have a homomorphism $\pi_1(E_L) \to \pi_1(S^4 \sm N(\Sigma)) \cong \Z$ and isomorphisms $\pi_1(W_{\pm}) \to \pi_1(S^4 \sm \Sigma) \cong \Z$, giving rise to $\Lambda$-coefficient systems.

\begin{lem}\label{lem:homology-of-W-pm}
  We have
\begin{enumerate}[font=\normalfont, label = (\roman*), ref= \roman*]
  \item $H_1(W_{\pm};\Lambda)=0 = H_3(W_{\pm};\Lambda)$;
  \item $H_2(W_{\pm};\Lambda) \cong \Lambda$;
  \item $H_2(W_{\pm},E_L;\Lambda) \cong \Lambda$;
  \item\label{it:last} $H_3(W_{\pm},E_L;\Lambda) =0 = H_4(W_{\pm},E_L;\Lambda)$;
\end{enumerate}
\end{lem}

\begin{proof}
  By the rising water principle, $W_{\pm}$ has a handle decomposition with one 0-handle, one 1-handle, and one 2-handle. The corresponding  handle chain complex with $\Lambda$ coefficients has the form
  \[
  \Lambda \xrightarrow{(0)} \Lambda \xrightarrow{(1-t)} \Lambda,
  \]
  supported in degrees $2$, $1$, and $0$,
  and hence $H_1(W_{\pm};\Lambda) = 0 = H_3(W_{\pm};\Lambda)$ and  $H_2(W_{\pm};\Lambda) \cong \Lambda$.

  Turning the Morse function upside down, we obtain a handle decomposition for $W_{\pm}$ relative to $E_L$ consisting of one 2-handle, one 3-handle, and one 4-handle. The handle chain complex for the pair $(W_{\pm},E_L)$ has the form
  \begin{equation}\label{eqn:chain-complex}
  \Lambda \xrightarrow{(x)} \Lambda \xrightarrow{(y)} \Lambda,
  \end{equation}
  supported in degrees $4$, $3$, and $2$, for some $x$, $y \in \Lambda$ with $xy=0$.
We will return to this chain complex after proving~\eqref{it:last}.
Write \[\partial_{\pm} := \ol{\partial W_{\pm} \sm E_L}.\]
By Poincar\'{e}--Lefschetz duality,
\[
H_3(W_{\pm},E_L; \Lambda) \cong H^1(W_{\pm},\partial_{\pm};\Lambda) \text{ and }
H_4(W_{\pm},E_L; \Lambda) \cong H^0(W_{\pm},\partial_{\pm};\Lambda).
\]
Since $H_1(W_{\pm};\Lambda)=0$, the long exact sequence of the pair $(W_{\pm},\partial_{\pm})$ ends with
\[0 \to H_1(W_{\pm},\partial_{\pm};\Lambda) \to H_0(\partial_{\pm};\Lambda) \xrightarrow{\cong} H_0(W_{\pm};\Lambda) \to H_0(W_{\pm},\partial_{\pm};\Lambda) \to 0.\]
It follows that $H_1(W_{\pm},\partial_{\pm};\Lambda)=0 = H_0(W_{\pm},\partial_{\pm};\Lambda)$.
So, by the universal coefficient spectral sequence, which gives a short exact sequence for $H^1$, we deduce that $H^1(W_{\pm},\partial_{\pm};\Lambda)=0 = H^0(W_{\pm},\partial_{\pm};\Lambda)$ and hence $H_3(W_{\pm},E_L; \Lambda) = 0 = H_4(W_{\pm},E_L; \Lambda)$.

Now we make use of the handle chain complex~\eqref{eqn:chain-complex}.
Since $H_4(W_{\pm},E_L; \Lambda)= \ker((x)) = 0$, we must have that $x \neq 0$. Since $xy=0$ by the chain complex condition, and since $\Lambda$ is an integral domain, we see that $y=0$.
Then, from the chain complex~\eqref{eqn:chain-complex}, we compute that  $H_2(W_{\pm},E_L;\Lambda) \cong \Lambda$.
\end{proof}

The previous proof is the key place where we used the assumption that $\Delta_L=0$.

\begin{lem}\label{lem:homology-of-EL}
  We have $H_1(E_L;\Lambda) \cong \Lambda \cong H_2(E_L;\Lambda)$, and $H_2(E_L;\Lambda) \to H_2(W_{\pm};\Lambda)$ is an isomorphism.
\end{lem}

\begin{proof}
 By \cref{lem:homology-of-W-pm}, the long exact sequence of the pair $(W_{\pm},E_L)$ with $\Lambda$-coefficients has the form
 \[0 \to H_2(E_L;\Lambda) \to \Lambda \xrightarrow{A}  \Lambda \to H_1(E_L;\Lambda) \to 0.\]
The map $A$ presents $H_1(E_L;\Lambda)$, and so is multiplication by the (single-variable) Alexander polynomial of $L$.  Since we assumed $\Delta_L=0$, it follows that $A$ is the trivial map and $H_1(E_L;\Lambda) \cong \Lambda \cong H_2(E_L;\Lambda)$. Furthermore, $H_2(E_L;\Lambda) \to H_2(W_{\pm};\Lambda)$ is an isomorphism, as claimed.
\end{proof}

Now we consider the Mayer--Vietoris sequence associated with the decomposition
\[
E_\Sigma:= S^4 \sm N(\Sigma) = W_+ \cup_{E_L} W_-.
\]
By Lemmas~\ref{lem:homology-of-W-pm} and \ref{lem:homology-of-EL}, we have
\[H_2(E_L;\Lambda)\cong \Lambda \xrightarrow{(1,1)} \Lambda \oplus \Lambda \to H_2(E_\Sigma;\Lambda) \to \Lambda \cong H_1(E_L;\Lambda) \to 0.\]
This yields  a short exact sequence
\begin{equation}\label{eqn:ses-giving-basis}
  0 \to \Lambda \to H_2(E_\Sigma;\Lambda) \to \Lambda \to 0.
\end{equation}
By \cref{prop:homology-closed}, we already knew that $H_2(E_\Sigma;\Lambda) \cong \Lambda \oplus \Lambda$, but this short exact sequence gives us a basis $\{[S_0], [S_1]\}$, as follows.
Let $S_0 \subseteq E_L$ denote a surface generating $H_2(E_L;\Lambda)$.
Let $\lambda$ be the $\Lambda$-intersection form on $H_2(E_\Sigma; \Lambda)$.
The intersection number $\lambda([S_0],[S_0])$ is trivial,  because we can use the boundary collar $E_L \times [0,1] \subseteq W_-$ to obtain a disjoint copy of $S_0$ that represents the same class in $H_2(E_\Sigma;\Lambda)$.

We define an immersed sphere $S_1$ as follows. Take a loop $\gamma$ in $E_L$ generating $H_1(E_L;\Lambda)$. Since $\gamma$ represents an element of $H_1(E_L;\Lambda)$, it represents an element of $\pi_1(E_L)$ in the kernel of the representation $\pi_1(E_L) \to \Z$ defining the $\Lambda$-coefficients, i.e.\ the total linking number $\lk(\gamma,L_1) + \lk(\gamma,L_2)=0$. The representation $\pi_1(E_L) \to \Z$ coincides with the inclusion-induced maps $\pi_1(E_L) \to \pi_1(W_{\pm}) \cong \Z$, hence $\gamma$ is null-homotopic in $W_{\pm}$. Choose a null-homotopy of $\gamma$ in $W_+$ and a null-homotopy of $\gamma$ in $W_-$, and glue the null-homotopies together. The sphere $S_1$ represents the image of a splitting of the short exact sequence~\eqref{eqn:ses-giving-basis}.

The intersection pairing $\lambda \colon H_2(E_\Sigma;\Lambda) \times H_2(E_\Sigma;\Lambda) \to \Lambda$ is represented with respect to the basis $\{[S_0],[S_1]\}$ for $H_2(E_\Sigma;\Lambda) \cong \Lambda \oplus \Lambda$  by a matrix of the form
		\[
		\begin{pmatrix}
			0 & b \\
			\ol{b}  & d
		\end{pmatrix}.
		\]
We have that
\[H_1(\partial E_\Sigma; \Lambda) \cong H_1(\Sigma \times S^1;\Lambda) \cong H_1(\Sigma \times \R;\Z) \cong  (\Lambda/(1-t))^2\] and $H_1(E_\Sigma;\Lambda)=0$ by parts \eqref{it:H1-lambda} and \eqref{it:boundary-closed} of \cref{prop:homology-closed}. So, the long exact sequence of the pair $(E_\Sigma,\partial E_{\Sigma})$ gives
\[H_2(E_\Sigma;\Lambda) \to H_2(E_\Sigma,\partial E_{\Sigma};\Lambda) \to (\Lambda/(1-t))^2 \to 0.\]
The intersection pairing $\lambda$ is adjoint to the composition
\[H_2(E_\Sigma;\Lambda) \to H_2(E_\Sigma,\partial E_{\Sigma};\Lambda) \xrightarrow{\cong} H^2(E_\Sigma;\Lambda) \xrightarrow{\cong} \Hom_\Lambda(H_2(E_\Sigma;\Lambda),\Lambda)\]
given by the map from the long exact sequence of the pair, Poincar\'{e} duality, and universal coefficients. Hence, the adjoint of the intersection form gives a presentation for  $H_1(\partial E_{\Sigma};\Lambda) \cong (\Lambda/(1-t))^2$.
The order of this $\Lambda$-module is $(t-1)^2$, and so, up to units, $b\cdot \ol{b}$ agrees with $(1-t)^2$. Thus $b = \pm t^k(1-t)$ for some $k \in \Z$. After a change of basing path and orientation choice for $[S_0]$, we can assume that $b= 1-t$, and so $\lambda$ is represented by
\[
		\begin{pmatrix}
			0 & 1-t \\
			1-t^{-1}  & d
		\end{pmatrix}
\]
for some $d \in \Lambda$.
It now follows by \cref{lemma:main-strategy} that $\Sigma$ is topologically unknotted, which concludes the proof of \cref{thm:morse}.
 \end{proof}

In the special case that the link $L$ appearing as $\Sigma \cap S^3$ is split, we will prove next that $\Sigma$ is moreover smoothly unknotted.
Note that if $L$ is split, then $\Delta_L(t)=0$, but there are links, e.g.\ 2-component boundary links, that are not split but that have vanishing Alexander polynomial.
 As $L$ is split, a theorem of Scharlemann~\cite[Main Theorem]{Scharlemann1985} implies that $L$ is the unlink and the bands are both trivial. The key to the proof will be the following lemma.

\begin{lem}\label{lemma:trivial-band}
  Consider a standard, oriented, 2-component unlink $L$ and a band $B$ with one end attached to each component of $L$, and suppose that $L \cup B$ is isotopic to the standard oriented unlink $L$ with a standard band $B'$ $($see the top of Figure~\ref{fig:band}$)$.  Then $B$ is isotopic to $B'$ relative to its ends, via an isotopy that fixes $L$ setwise throughout.
\end{lem}

\begin{proof}
If the isotopy from $L \cup B$ to $L \cup B'$ switches the components of $L$, we concatenate this isotopy with a $\pi$-rotation of $L \cup B'$ around an axis perpendicular to the plane of $L \cup B'$ through the midpoint of $B'$, which preserves $B'$. So, we can assume that, after the isotopy, $L$ is preserved as an ordered, oriented link.

  Consider the isotopy of $L \cup B$ to $L \cup B'$, and then forget the band. We obtain an isotopy of $L$ to itself. The group of such self-isotopies, considered up to isotopy of isotopy, is called the \emph{motion group} of $L$.  We are thus given an element $\gamma$ of the motion group, which we can extend to an ambient isotopy $\Gamma_t \colon S^3 \to S^3$ for $t \in I$, such that $\Gamma_0=\Id_{S^3}$,  $\Gamma_1(L)=L$ as an ordered, oriented link, and $\Gamma_1(B) = B'$.

\begin{figure}
	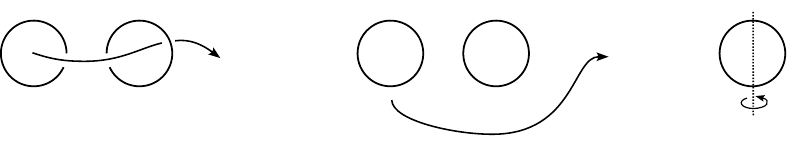
	\caption{The generators $\rho$, $\sigma$, and $\tau_i$ for $i \in \{1,2\}$ of the motion group of a 2-component link.}
	\label{fig:motion-generators}
\end{figure}

  The motion group of an unlink was shown by Goldsmith~\cite{Goldsmith} to be generated by standard isotopies; see also Brendle--Hatcher~\cite{Brendle-Hatcher} and Damiani--Kamada~\cite[Proposition~3.4]{Damiani-Kamada}. In the case of two components $L_1$ and $L_2$, there are four standard isotopies, called $\rho$, $\sigma$, $\tau_1$, and $\tau_2$. Both $\rho$ and $\sigma$ switch the positions of $L_1$ and $L_2$.   During $\rho$, the component $L_1$ passes through $L_2$ while swapping $L_1$ and $L_2$. During~$\sigma$, the components $L_1$ and $L_2$ pass past each other. The isotopy $\tau_i$ rotates the $i$th component by angle $\pi$ around an axis connecting a pair of antipodal points of $L_i$. See \cite[Figure~1]{Brendle-Hatcher}, which we have reproduced in Figure~\ref{fig:motion-generators}, except we have isotoped $\sigma$ relative to its ends to better suit the rest of the proof.  The motion group of $L$ is generated by $\{\rho, \sigma, \tau_1, \tau_2\}$. Let $w$ be a word in these generators representing the motion $\gamma$.  As $\tau_1\rho = \rho \tau_2$ and $\tau_1\sigma = \sigma \tau_2$, we can move every copy of $\tau_1$ and $\tau_2$ to the end of $w$. Using the relations $\tau_1\tau_2 = \tau_2\tau_1$, $\tau_1^2 = e$, and $\tau_2^2 = e$, we can assume that $w$ ends with $\tau_1^{\varepsilon_1} \tau_2^{\varepsilon_2}$ for $\varepsilon_1$, $\varepsilon_2 \in \{0, 1\}$. We have $\varepsilon_1 = 0$ and $\varepsilon_2 = 0$ because $L$ is oriented and the motion $\iota$ is orientation-preserving. So $\tau_1$ and $\tau_2$ can be ignored, and we can assume $w$ is a word in $\rho$ and $\sigma$.

The motion $\gamma^{-1}$ extends to an isotopy of diffeomorphisms $\Upsilon_t \colon S^3 \to S^3$ starting at $\Upsilon_0=\Id_{S^3}$. To obtain this, write $\gamma^{-1}$ as a word in the generators $\rho$, $\sigma$, and $\rho^{-1}$  (note that $\sigma^{-1}= \sigma$).  Apply isotopy extension to each to obtain the path of diffeomorphisms $\Upsilon_t$.   We have that $\Upsilon_1(L) = L$ as an ordered, oriented link. Using the simple form of the generators of the motion group, observe that $\Upsilon_1(B')$ is isotopic to $B'$, relative to $L$. More specifically, after applying $\rho$ or $\sigma$ to the standard band shown in the top row of Figure~\ref{fig:band}, we obtain the band shown in the second row of Figure~\ref{fig:band}. We then rotate each link component in the direction of the arrows in the figure through angle $\pi$, after which the band is isotopic to the standard band, keeping its ends fixed.
\begin{figure}
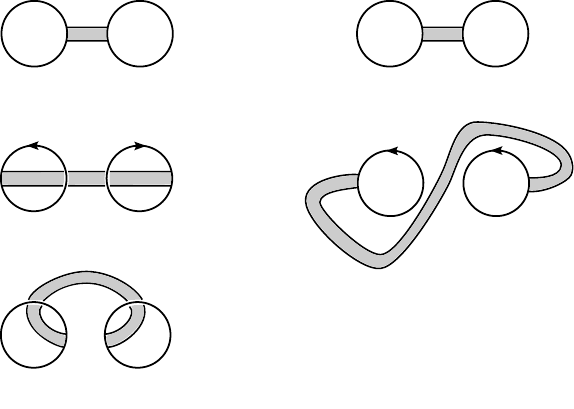
\caption{The second row shows the effect of $\rho$ (left) and $\sigma$ (right) on the standard band shown in the top row. The band can then be isotoped back to the standard position by rotating each link component through angle $\pi$ in the direction of the arrow.}
\label{fig:band}
\end{figure}
This observation produces an isotopy $\Delta_t \colon S^3 \to S^3$ that fixes $L$ setwise for all $t \in I$, but sends $\Upsilon_1(B')$ back to the trivial band $B'$.  Concatenate $\Upsilon_t$ and $\Delta_t$ to obtain an ambient isotopy $\Upsilon_t \cdot \Delta_t$ representing $\gamma^{-1}$ in the motion group of $L$ that sends $B'$ to $B'$.  Now concatenate $\Gamma_t$ with $\Upsilon_t \cdot \Delta_t$ to obtain an ambient isotopy such that the motion of $L$ represents $\gamma\cdot \gamma^{-1} = e$ in the motion group and sends $B$ to $B'$. Isotope this isotopy relative to its endpoints to one where $L$ is fixed setwise for all time, which we may do since, when restricted to $L$, this gives the trivial element of the motion group. We have constructed the desired isotopy from $B$ to $B'$ that fixes~$L$ setwise throughout.
\end{proof}

\four*

\begin{proof}
  A theorem of Scharlemann~\cite[Main Theorem]{Scharlemann1985} implies, as mentioned before, that $L$ is the unlink and the bands are both trivial. By isotoping $\Sigma$ in a bicollar of the equatorial $S^3$, we may and shall assume that $L$ is in the position of the standard unlink. Let $B_{\pm}$ denote the band that determines the surface $\Sigma_{\pm} := \Sigma \cap S^4_{\pm}$, up to isotopy. Since the band is trivial, $L \cup B_{\pm}$ is isotopic to $L\cup B'$, where $B'$ is the standard band. Then, by \cref{lemma:trivial-band}, the band $B_{\pm}$ is isotopic relative to $L$ to $B'$, through an isotopy that fixes $L$ setwise. In the theory of banded unlink representations of surfaces introduced by Yoshikawa~\cite{Yoshikawa}, this corresponds to a smooth isotopy relative to boundary from $\Sigma_{\pm}$ to the standard annulus in $S^4_{\pm}$ with boundary~$L \subseteq S^3 = \partial S^4_{\pm}$. The union of two standard annuli is a smoothly unknotted torus, so the proof is complete.
\end{proof}

\section{Double branched covers}\label{sec:dbc}
	
	Given a surface $S$ in a 4-manifold $X$, we denote the double cover of $X$ branched along $S$ by $\Sigma(X, S)$. As in the case of the exotic copies of $5 \RP^2$,  $\# 6 \RP^2$, and $\# 10 \RP^2$ constructed by Mati\'c--\"{O}zt\"{u}rk--Stipsicz~\cite{matic2023exotic},  Finashin~\cite{Finashin}, and Finashin--Kreck--Viro~\cite{FKV}, respectively, one might try to distinguish the turned 1-twist-spun torus $T_{K,1,1}$ and the standard torus $T^2$ using their double branched covers. Note that $\Sigma(S^4, T^2)$ is diffeomorphic to $S^2 \times S^2$, where the covering involution is given by reflection in $S^1$ in both $S^2$ factors.

	\begin{prop}\label{prop:DBC}
		Let $K$ be a knot in $D^3$. Then $\Sigma(S^4, T_{K,1,1})$ is diffeomorphic to $(S^2 \times S^2) \# X$ for a homotopy 4-sphere $X$.
	\end{prop}

	\begin{proof}
		Using the notation of Definition~\ref{def:torus}, note that both $\phi_j$ and $\eta_i|_{S^2 \times \{\theta\}}$ for $\theta \in S^1$ fix the $z$-axis pointwise. Let
		\[
		J := \{0,0\} \times [1/2, 1] \subseteq D^3.
		\]
		Isotope $K$ such that $K \cap J = \{(0,0,1/2)\}$. Then
		\[
		D := \frac{J \times I}{\{(x,0) = (\phi_1(x), 1)  :  x \in J\}} \cup_{\eta_1} (\{(0,0,1)\} \times D^2)
		\]
		is a disc in $S^4$ with boundary on $T_{K,1,1}$. Furthermore, consider the sphere
		\[
		S_0 := \partial D^3 \times \{0\} \subseteq S^4
		\]
		disjoint from $K \times \{0\}$. Then $S_0$ is disjoint from the branch locus $T_{K,1,1}$ and is simply connected, so it lifts to a pair of spheres in $\Sigma(S^4, T_{K,1,1})$, one of which we denote by $S$. Since~$S_0$ has self-intersection zero, so does $S$. As $D \cap T_{K,1,1} = \partial D$, the disc $D$ lifts to a sphere $S'$ in $\Sigma(S^4, T_{K,1,1})$, which has even self-intersection (in fact, it is four by the proof of Theorem~\ref{thm:Boyle}). Note that $\partial D^3 \cap J = \{(0,0,1)\}$, so $|S_0 \cap D| = 1$. It follows that $|S \pitchfork S'| = 1$, so $N(S \cup S')$ is a plumbing of two spheres of self-intersection zero and four, respectively, and is hence diffeomorphic to $(S^2 \times S^2) \setminus D^4$; see Neumann--Weintraub~\cite{plumbings}.
		
		As $T_{K,1,1}$ and $T^2$ are topologically isotopic by Theorem~\ref{thm:Boyle} and $\Sigma(S^4, T^2) \cong S^2 \times S^2$, it follows that~$\Sigma(S^4, T_{K,1,1})$ is homeomorphic to $S^2 \times S^2$. Consequently,
		\[
		\Sigma(S^4, T_{K,1,1})  \setminus N(S \cup S')
		\]
		is a homotopy 4-disc, and the result follows.
	\end{proof}

Let $M$ be a closed, connected, smooth, and oriented 4-manifold $M$ with $b_1(M) = 0$ and $b_2^+(M) > 1$. Furthermore, let $o$ be an orientation of $H^2_+(M; \R)$, which is called a \emph{homology orientation}. The Seiberg--Witten invariant of $M$ is a map
\[
\SW_{M,o} \colon \Spin^c(M) \to \Z,
\]
where $\Spin^c(M)$ is the set of $\Spin^c$ structures on $X$, which is an affine space modelled on $H^2(M)$. This satisfies $\SW_{M,-o} = -\SW_{M,o}$. If $\Phi \colon M' \to M$ is a diffeomorphism, $o$ is a homology orientation of $M$, and $\fs \in \Spin^c(M)$, then
\[
\SW_{M', \Phi^* o}(\Phi^* \fs) = \SW_{M, o}(\fs).
\]
We say that the Seiberg--Witten invariants of $M$ and $M'$ \emph{agree} if there is an affine isomorphism $\phi \colon \Spin^c(M) \to \Spin^c(M')$ such that
\[
\SW_{M',o'} \circ \phi = \pm \SW_{M,o}
\]
for arbitrary homology orientations $o$ of $M$ and $o'$ of $M'$.

By Kronheimer and Mrowka~\cite{KM-Thom} (see also Gompf--Stipsicz~\cite[Theorem~2.4.8]{Gompf-Stipsicz}), if $\Sigma$ is a smoothly embedded, connected, and oriented surface of genus $g$ in $M$ with self-intersection $[\Sigma]^2 \ge 0$ and $[\Sigma] \neq 0$ in $H_2(M)$, then
\begin{equation}\label{eqn:adjunction}
2g -2 \ge [\Sigma]^2 + |\langle c_1(\fs), [\Sigma] \rangle|
\end{equation}
whenever $\SW_{M, o}(\fs) \neq 0$. This is known as the \emph{generalised andjunction formula}.

One can also define $\SW_{M,o}(\fs) \in \Z$ when $b_1(X) = 0$ and $b_2^+(M) = 1$, but it also depends on a choice of two possible metric chambers $c_+$ and $c_-$. When the expected dimension of the Seiberg--Witten moduli space
\[
d(\fs) := \left(c_1(\fs)^2 - 2\chi(M) - 3\sigma(M) \right)/4 \in 2\Z_{\ge 0},
\]
then the dependence of the Seiberg--Witten invariant on the choice of chamber is described by the \emph{wall-crossing formula}
\begin{equation}\label{eqn:wall-crossing}
\SW_{M,o}(\fs, c_-) = \SW_{M,o}(\fs, c_+) + (-1)^{d(\fs)/2}
\end{equation}
due to Kronheimer and Mrowka~\cite{KM-Thom}.
We have the following consequence of Proposition~\ref{prop:DBC}.

\begin{cor}
Let $K$ be a knot in $S^3$. Then the Seiberg--Witten invariants of $\Sigma(S^4, T_{K,1,1})$ and $S^2 \times S^2$ agree.
\end{cor}

	\begin{proof}
	For a closed, connected, smooth, and oriented 4-manifold $M$ with $b_2^+(M) = 1$ and $b_2^-(M) < 10$, there is a canonical metric chamber for the Seiberg--Witten invariant; see Szab\'o~\cite[Lemma~3.2]{Szabo}. Furthermore, the generalised adjunction formula holds with this choice. Hence, if $M$ is $(S^2 \times S^2) \# X$ for a homotopy 4-sphere $X$, then, for every $\fs \in \Spin^c(M)$, we have $\SW_{M, o}(\fs, c) = 0$ for the canonical metric chamber $c$ by the generalised adjunction formula~\eqref{eqn:adjunction} applied to the self-intersection zero, homologically non-trivial 2-sphere $\Sigma := S^2 \times \{p\} \subseteq M$ for $p \in S^2$ (as the inequality $-2 \ge |\langle c_1(\fs), [\Sigma] \rangle|$ does not hold for any $\fs$). Hence, the Seiberg--Witten invariants of $\Sigma(S^4, T_{K,1,1})$ and $S^2 \times S^2$ both vanish in the canonical metric chambers, and also agree in the other chamber by the wall-crossing formula~\eqref{eqn:wall-crossing}.
	\end{proof}

	In particular, we cannot use the Seiberg--Witten invariants of the double branched covers to distinguish $T_{K,1,1}$
 from $T^2$ up to diffeomorphism of pairs. However, this might still be a possibility if $T$ has four critical points and $T \cap S^3$ is not split, but we do not have any such examples that are not already known to be smoothly standard.
	
	We finally consider the double cover $M$ of $S^4$ branched along the Cochran--Davis Seifert surface union a ribbon disc.	By Theorem~\ref{thm:ribbon}, the 4-manifold $M$ is homeomorphic to $S^2 \times S^2$. On the other hand, we have the following.
	
\begin{prop}
	Let $M$ be the double cover of $S^4$ branched along the surface $\Sigma$ constructed in Theorem~\ref{thm:ribbon}. Then $M$ can be obtained from $S^2 \times S^2$ using knot surgery along a 0-homologous torus.
\end{prop}

\begin{proof}
	We use the notation of Section~\ref{sec:Cochran-Davis-examples}. Let us write
	\[
	S^4_\pm := \frac{S^3 \times I_\pm}{S^3 \times \{\pm 1\}},
	\]
	where $I_+ = [0,1]$ and $I_- = [-1,0]$.
	Let the Seifert surface pushed into $S^4_-$ be
	\[
	\digamma' := (R \times [-1/8, 0]) \cup (F' \times \{-1/8\}).
	\]
	\begin{enumerate}[font=\normalfont, label = (\roman*), ref= \roman*]
		 \item \label{it:pushed-in-A}   We can arrange that $A_- \cap (S^3 \times [-1/2,0])$ takes the form $(\eta_1 \sqcup \eta_2) \times [-1/2,0]$, i.e.\  $A_-$ is a product in this collar.
		\item  We assume that the saddle moves used to construct $A_{-}$ are performed in $S^3 \times [-3/4,-1/2]$, and the disc capping it off (i.e.\ the minimum of the projection to $[-1,0]$ restricted to $A_-$) lies in $S^3 \times [-1,-3/4]$.
		\item   We suppose that $A_+ \cap (S^3 \times [0,1/2])$ takes the form $(\eta_1 \sqcup \eta_2) \times [0,1/2]$.
		\item  We assume that the saddle moves used to construct $A_{+}$ and $\Delta$ are performed in $S^3 \times [1/2,3/4]$, and the discs capping them off (i.e.\ the maxima of the projection to $[0,1]$ restricted to $A_+ \cup \Delta$) lie in $S^3 \times [3/4,1]$.
	\end{enumerate}

	Consider the torus $T_A := A_+ \cup A_-$ in $S^4 \setminus \Sigma'$, where
	\[
	\Sigma' = \Delta \cup \digamma' \subseteq S^4_+ \cup S^4_-.
	\]
	Let $p \colon M \to S^2 \times S^2$ be the covering map. We now show that $T := p^{-1}(T_A)$ is connected, and hence a torus. The curve $\eta_1$ is a meridian of $T_A$. Since $\lk(\eta_1, R) = 0$ and $\Sigma' \cap S^3 = R$, the meridian of $T_A$ is 0-homologous in $S^4 \setminus \Sigma'$.
	
	We claim that a longitude of $T_A$ represents a generator of $H_1(S^4 \setminus \Sigma') \cong \Z$. Recall that the annulus~$A_\pm$ is obtained by performing a saddle move on $\eta_1 \cup \eta_2$ and capping off the resulting unknot with a disc. Suppose that the saddle is attached at points $x_1 \in \eta_1$ and $x_2 \in \eta_2$; see Figure~\ref{fig:arc}. We can assume that $x_1$ and $x_2$ lie on opposite sides of the Seifert surface $F'$. Suppose that the saddle point of~$A_\pm$ lies in $S^3 \times \{t_\pm\}$, and let $c_\pm$ be the core of the band that is attached. Then
	\[
	\ell := \bigl(\{x_1, x_2\} \times [t_-, t_+] \bigr) \cup c_+ \cup c_-
	\]
	is a longitude of $T_A$.
	\begin{figure}
		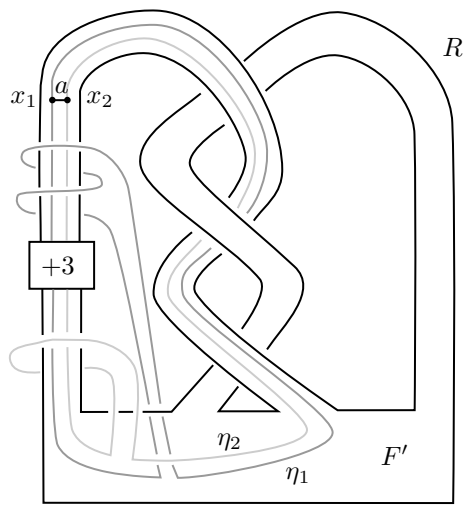
		\caption{The points $x_1 \in \eta_1$ and $x_2 \in \eta_2$, together with the arc $a$ connecting them that intersects $F'$ in a single point.}
		\label{fig:arc}
	\end{figure}
	Let $a$ be a straight arc in $S^3$ connecting $x_1$ and $x_2$ such that $|a \pitchfork F'| = 1$; see Figure~\ref{fig:arc}. Then $D_A := a \times [t_-, t_+]$ gives a disc in $S^4$ with boundary $\ell$. We have $D_A \pitchfork \digamma' = \{(a \cap F', -1/8)\}$.
	Furthermore, $D_A \cap \Delta = \emptyset$, because we attach the saddle to $\eta_1 \cup \eta_2$ to obtain $A_+$ before we attach the saddle to $R$ when forming $\Delta$. It follows that $|D_A \pitchfork \Sigma'| = 1$. Hence $\ell$ represents a generator of $H_1(S^4 \setminus \Sigma')$, as claimed. Consequently,  the pre-image $p^{-1}(T_A)$ is connected.
	
	By construction, we obtain $\Sigma$ from $\Sigma'$ by performing knot surgery along $T_A$ with pattern some knot~$J$ in $S^3$ (in the main Cochran--Davis example, $J$ is the left-handed trefoil). So we can obtain $M$ from $S^2 \times S^2$ by knot surgery along the torus $T$. Since $H_2(S^4) = 0$, the torus $T_A$ bounds a 3-manifold~$N$ in $S^4$. Make $N$ transverse to $\Sigma'$. Then $p^{-1}(N)$ is a 3-manifold in $M$ with boundary $T$. Hence $T$ is 0-homologous, as claimed.
\end{proof}

The knot surgery formula of Fintushel and Stern~\cite{FSKnotSurgery} only applies when the torus along which we surger is homologically essential. Furthermore, $S^2 \times S^2$ has vanishing Seiberg--Witten invariants in the relevant chamber. Hence, we cannot use the knot surgery formula to distinguish $M$ from $S^2 \times S^2$, and consequently $\Sigma$ from $T^2$.

	\bibliographystyle{alpha}
	\bibliography{topology}
	
\end{document}

%% file: fig1.pdf_tex
\begingroup%
  \makeatletter%
  \providecommand\color[2][]{%
    \errmessage{(Inkscape) Color is used for the text in Inkscape, but the package 'color.sty' is not loaded}%
    \renewcommand\color[2][]{}%
  }%
  \providecommand\transparent[1]{%
    \errmessage{(Inkscape) Transparency is used (non-zero) for the text in Inkscape, but the package 'transparent.sty' is not loaded}%
    \renewcommand\transparent[1]{}%
  }%
  \providecommand\rotatebox[2]{#2}%
  \newcommand*\fsize{\dimexpr\f@size pt\relax}%
  \newcommand*\lineheight[1]{\fontsize{\fsize}{#1\fsize}\selectfont}%
  \ifx\svgwidth\undefined%
    \setlength{\unitlength}{249.38247687bp}%
    \ifx\svgscale\undefined%
      \relax%
    \else%
      \setlength{\unitlength}{\unitlength * \real{\svgscale}}%
    \fi%
  \else%
    \setlength{\unitlength}{\svgwidth}%
  \fi%
  \global\let\svgwidth\undefined%
  \global\let\svgscale\undefined%
  \makeatother%
  \begin{picture}(1,0.42130904)%
    \lineheight{1}%
    \setlength\tabcolsep{0pt}%
    \put(0,0){\includegraphics[width=\unitlength,page=1]{fig1.pdf}}%
    \put(0.16912684,0.08982083){\makebox(0,0)[lt]{\lineheight{1.25}\smash{\begin{tabular}[t]{l}$\Sigma$\end{tabular}}}}%
    \put(0.11046062,0.25416312){\makebox(0,0)[lt]{\lineheight{1.25}\smash{\begin{tabular}[t]{l}$D$\end{tabular}}}}%
    \put(0.26568318,0.09144771){\makebox(0,0)[lt]{\lineheight{1.25}\smash{\begin{tabular}[t]{l}$T$\end{tabular}}}}%
    \put(0.23349773,0.38111689){\makebox(0,0)[lt]{\lineheight{1.25}\smash{\begin{tabular}[t]{l}$T$\end{tabular}}}}%
    \put(0.94157808,0.38648117){\makebox(0,0)[lt]{\lineheight{1.25}\smash{\begin{tabular}[t]{l}$T$\end{tabular}}}}%
    \put(0.76619191,0.316746){\makebox(0,0)[lt]{\lineheight{1.25}\smash{\begin{tabular}[t]{l}$+$\end{tabular}}}}%
    \put(0.90045216,0.32211024){\makebox(0,0)[lt]{\lineheight{1.25}\smash{\begin{tabular}[t]{l}$-$\end{tabular}}}}%
    \put(0.96661106,0.09144771){\makebox(0,0)[lt]{\lineheight{1.25}\smash{\begin{tabular}[t]{l}$S$\end{tabular}}}}%
    \put(0.83949538,0.24164654){\makebox(0,0)[lt]{\lineheight{1.25}\smash{\begin{tabular}[t]{l}$\delta$\end{tabular}}}}%
    \put(0.87044688,0.08982083){\makebox(0,0)[lt]{\lineheight{1.25}\smash{\begin{tabular}[t]{l}$\Sigma$\end{tabular}}}}%
  \end{picture}%
\endgroup%

%% file: fig3.pdf_tex
\begingroup%
  \makeatletter%
  \providecommand\color[2][]{%
    \errmessage{(Inkscape) Color is used for the text in Inkscape, but the package 'color.sty' is not loaded}%
    \renewcommand\color[2][]{}%
  }%
  \providecommand\transparent[1]{%
    \errmessage{(Inkscape) Transparency is used (non-zero) for the text in Inkscape, but the package 'transparent.sty' is not loaded}%
    \renewcommand\transparent[1]{}%
  }%
  \providecommand\rotatebox[2]{#2}%
  \newcommand*\fsize{\dimexpr\f@size pt\relax}%
  \newcommand*\lineheight[1]{\fontsize{\fsize}{#1\fsize}\selectfont}%
  \ifx\svgwidth\undefined%
    \setlength{\unitlength}{376.02144021bp}%
    \ifx\svgscale\undefined%
      \relax%
    \else%
      \setlength{\unitlength}{\unitlength * \real{\svgscale}}%
    \fi%
  \else%
    \setlength{\unitlength}{\svgwidth}%
  \fi%
  \global\let\svgwidth\undefined%
  \global\let\svgscale\undefined%
  \makeatother%
  \begin{picture}(1,0.49604343)%
    \lineheight{1}%
    \setlength\tabcolsep{0pt}%
    \put(0,0){\includegraphics[width=\unitlength,page=1]{fig3.pdf}}%
    \put(0.41497288,0.44953622){\makebox(0,0)[lt]{\lineheight{1.25}\smash{\begin{tabular}[t]{l}$R$\end{tabular}}}}%
    \put(0.22451463,0.00927512){\makebox(0,0)[lt]{\lineheight{1.25}\smash{\begin{tabular}[t]{l}$F'$\end{tabular}}}}%
    \put(0.01265965,0.230466){\makebox(0,0)[lt]{\lineheight{1.25}\smash{\begin{tabular}[t]{l}$+3$\end{tabular}}}}%
    \put(0,0){\includegraphics[width=\unitlength,page=2]{fig3.pdf}}%
    \put(0.33675876,0.02621813){\makebox(0,0)[lt]{\lineheight{1.25}\smash{\begin{tabular}[t]{l}$\alpha'$\end{tabular}}}}%
    \put(0,0){\includegraphics[width=\unitlength,page=3]{fig3.pdf}}%
    \put(0.97926953,0.44953622){\makebox(0,0)[lt]{\lineheight{1.25}\smash{\begin{tabular}[t]{l}$R$\end{tabular}}}}%
    \put(0.57695544,0.230466){\makebox(0,0)[lt]{\lineheight{1.25}\smash{\begin{tabular}[t]{l}$+3$\end{tabular}}}}%
    \put(0,0){\includegraphics[width=\unitlength,page=4]{fig3.pdf}}%
    \put(0.82202793,0.02486018){\makebox(0,0)[lt]{\lineheight{1.25}\smash{\begin{tabular}[t]{l}$\eta_1$\end{tabular}}}}%
    \put(0.75341389,0.0591774){\makebox(0,0)[lt]{\lineheight{1.25}\smash{\begin{tabular}[t]{l}$\eta_2$\end{tabular}}}}%
    \put(0,0){\includegraphics[width=\unitlength,page=5]{fig3.pdf}}%
  \end{picture}%
\endgroup%

%% file: fig4.pdf_tex
\begingroup%
  \makeatletter%
  \providecommand\color[2][]{%
    \errmessage{(Inkscape) Color is used for the text in Inkscape, but the package 'color.sty' is not loaded}%
    \renewcommand\color[2][]{}%
  }%
  \providecommand\transparent[1]{%
    \errmessage{(Inkscape) Transparency is used (non-zero) for the text in Inkscape, but the package 'transparent.sty' is not loaded}%
    \renewcommand\transparent[1]{}%
  }%
  \providecommand\rotatebox[2]{#2}%
  \newcommand*\fsize{\dimexpr\f@size pt\relax}%
  \newcommand*\lineheight[1]{\fontsize{\fsize}{#1\fsize}\selectfont}%
  \ifx\svgwidth\undefined%
    \setlength{\unitlength}{171.20688157bp}%
    \ifx\svgscale\undefined%
      \relax%
    \else%
      \setlength{\unitlength}{\unitlength * \real{\svgscale}}%
    \fi%
  \else%
    \setlength{\unitlength}{\svgwidth}%
  \fi%
  \global\let\svgwidth\undefined%
  \global\let\svgscale\undefined%
  \makeatother%
  \begin{picture}(1,1.08933096)%
    \lineheight{1}%
    \setlength\tabcolsep{0pt}%
    \put(0,0){\includegraphics[width=\unitlength,page=1]{fig4.pdf}}%
    \put(0.95446969,0.98731582){\makebox(0,0)[lt]{\lineheight{1.25}\smash{\begin{tabular}[t]{l}$R$\end{tabular}}}}%
    \put(0.07086766,0.50617214){\makebox(0,0)[lt]{\lineheight{1.25}\smash{\begin{tabular}[t]{l}$+3$\end{tabular}}}}%
    \put(0,0){\includegraphics[width=\unitlength,page=2]{fig4.pdf}}%
  \end{picture}%
\endgroup%

%% file: fig9.pdf_tex
\begingroup%
  \makeatletter%
  \providecommand\color[2][]{%
    \errmessage{(Inkscape) Color is used for the text in Inkscape, but the package 'color.sty' is not loaded}%
    \renewcommand\color[2][]{}%
  }%
  \providecommand\transparent[1]{%
    \errmessage{(Inkscape) Transparency is used (non-zero) for the text in Inkscape, but the package 'transparent.sty' is not loaded}%
    \renewcommand\transparent[1]{}%
  }%
  \providecommand\rotatebox[2]{#2}%
  \newcommand*\fsize{\dimexpr\f@size pt\relax}%
  \newcommand*\lineheight[1]{\fontsize{\fsize}{#1\fsize}\selectfont}%
  \ifx\svgwidth\undefined%
    \setlength{\unitlength}{377.35685025bp}%
    \ifx\svgscale\undefined%
      \relax%
    \else%
      \setlength{\unitlength}{\unitlength * \real{\svgscale}}%
    \fi%
  \else%
    \setlength{\unitlength}{\svgwidth}%
  \fi%
  \global\let\svgwidth\undefined%
  \global\let\svgscale\undefined%
  \makeatother%
  \begin{picture}(1,0.21002405)%
    \lineheight{1}%
    \setlength\tabcolsep{0pt}%
    \put(0,0){\includegraphics[width=\unitlength,page=1]{fig9.pdf}}%
    \put(0.95517783,0.20241562){\makebox(0,0)[lt]{\lineheight{1.25}\smash{\begin{tabular}[t]{l}$i$\end{tabular}}}}%
    \put(0.09888769,0.0022877){\makebox(0,0)[lt]{\lineheight{1.25}\smash{\begin{tabular}[t]{l}$\rho$\end{tabular}}}}%
    \put(0.60668828,0.0022877){\makebox(0,0)[lt]{\lineheight{1.25}\smash{\begin{tabular}[t]{l}$\sigma$\end{tabular}}}}%
    \put(0.9492284,0.0022877){\makebox(0,0)[lt]{\lineheight{1.25}\smash{\begin{tabular}[t]{l}$\tau_i$\end{tabular}}}}%
  \end{picture}%
\endgroup%

%% file: fig10.pdf_tex
\begingroup%
  \makeatletter%
  \providecommand\color[2][]{%
    \errmessage{(Inkscape) Color is used for the text in Inkscape, but the package 'color.sty' is not loaded}%
    \renewcommand\color[2][]{}%
  }%
  \providecommand\transparent[1]{%
    \errmessage{(Inkscape) Transparency is used (non-zero) for the text in Inkscape, but the package 'transparent.sty' is not loaded}%
    \renewcommand\transparent[1]{}%
  }%
  \providecommand\rotatebox[2]{#2}%
  \newcommand*\fsize{\dimexpr\f@size pt\relax}%
  \newcommand*\lineheight[1]{\fontsize{\fsize}{#1\fsize}\selectfont}%
  \ifx\svgwidth\undefined%
    \setlength{\unitlength}{275.23757157bp}%
    \ifx\svgscale\undefined%
      \relax%
    \else%
      \setlength{\unitlength}{\unitlength * \real{\svgscale}}%
    \fi%
  \else%
    \setlength{\unitlength}{\svgwidth}%
  \fi%
  \global\let\svgwidth\undefined%
  \global\let\svgscale\undefined%
  \makeatother%
  \begin{picture}(1,0.71187923)%
    \lineheight{1}%
    \setlength\tabcolsep{0pt}%
    \put(0.14072,0.001107){\makebox(0,0)[lt]{\lineheight{1.25}\smash{\begin{tabular}[t]{l}$\rho$\end{tabular}}}}%
    \put(0.76689387,0.18030773){\makebox(0,0)[lt]{\lineheight{1.25}\smash{\begin{tabular}[t]{l}$\sigma$\end{tabular}}}}%
    \put(0,0){\includegraphics[width=\unitlength,page=1]{fig10.pdf}}%
  \end{picture}%
\endgroup%

%% file: fig8.pdf_tex
\begingroup%
  \makeatletter%
  \providecommand\color[2][]{%
    \errmessage{(Inkscape) Color is used for the text in Inkscape, but the package 'color.sty' is not loaded}%
    \renewcommand\color[2][]{}%
  }%
  \providecommand\transparent[1]{%
    \errmessage{(Inkscape) Transparency is used (non-zero) for the text in Inkscape, but the package 'transparent.sty' is not loaded}%
    \renewcommand\transparent[1]{}%
  }%
  \providecommand\rotatebox[2]{#2}%
  \newcommand*\fsize{\dimexpr\f@size pt\relax}%
  \newcommand*\lineheight[1]{\fontsize{\fsize}{#1\fsize}\selectfont}%
  \ifx\svgwidth\undefined%
    \setlength{\unitlength}{171.20685723bp}%
    \ifx\svgscale\undefined%
      \relax%
    \else%
      \setlength{\unitlength}{\unitlength * \real{\svgscale}}%
    \fi%
  \else%
    \setlength{\unitlength}{\svgwidth}%
  \fi%
  \global\let\svgwidth\undefined%
  \global\let\svgscale\undefined%
  \makeatother%
  \begin{picture}(1,1.08945964)%
    \lineheight{1}%
    \setlength\tabcolsep{0pt}%
    \put(0,0){\includegraphics[width=\unitlength,page=1]{fig8.pdf}}%
    \put(0.95446969,0.98731593){\makebox(0,0)[lt]{\lineheight{1.25}\smash{\begin{tabular}[t]{l}$R$\end{tabular}}}}%
    \put(0.07086767,0.50617224){\makebox(0,0)[lt]{\lineheight{1.25}\smash{\begin{tabular}[t]{l}$+3$\end{tabular}}}}%
    \put(0,0){\includegraphics[width=\unitlength,page=2]{fig8.pdf}}%
    \put(0.60912012,0.0546004){\makebox(0,0)[lt]{\lineheight{1.25}\smash{\begin{tabular}[t]{l}$\eta_1$\end{tabular}}}}%
    \put(0.45842319,0.12997125){\makebox(0,0)[lt]{\lineheight{1.25}\smash{\begin{tabular}[t]{l}$\eta_2$\end{tabular}}}}%
    \put(0,0){\includegraphics[width=\unitlength,page=3]{fig8.pdf}}%
    \put(0.00393745,0.88313973){\makebox(0,0)[lt]{\lineheight{1.25}\smash{\begin{tabular}[t]{l}$x_1$\end{tabular}}}}%
    \put(0.17116132,0.88313973){\makebox(0,0)[lt]{\lineheight{1.25}\smash{\begin{tabular}[t]{l}$x_2$\end{tabular}}}}%
    \put(0,0){\includegraphics[width=\unitlength,page=4]{fig8.pdf}}%
    \put(0.10111545,0.90658069){\makebox(0,0)[lt]{\lineheight{1.25}\smash{\begin{tabular}[t]{l}$a$\end{tabular}}}}%
    \put(0.81866815,0.08686969){\makebox(0,0)[lt]{\lineheight{1.25}\smash{\begin{tabular}[t]{l}$F'$\end{tabular}}}}%
  \end{picture}%
\endgroup%